\documentclass[final]{siamltex1213}

\usepackage{amsmath,amstext,amsbsy,amssymb,graphicx,mathdots,todonotes}

\newcommand{\lam}{{\lambda}}

\newcommand{\B}{{\mathcal{B}}}

\newenvironment{dedication}
        {\vspace{1ex}\begin{quotation}\begin{center}\begin{em}}
        {\par\end{em}\end{center}\end{quotation}\vspace{1ex}}

\DeclareMathOperator{\modulo}{mod}

\newtheorem{remark}[theorem]{Remark}
\newtheorem{example}[theorem]{Example}

\def\css{{\,\boxplus\hspace{-2.2mm}\raisebox{0.3mm}{$\rightarrow$} @}}
\def\rss{{\,\boxplus \hspace{-2.2mm}\raisebox{-1mm}{$\downarrow$} \;}}
\mathcode`@="8000
{\catcode`\@=\active\gdef@{\mkern1mu}}

\def\bezoutian{B\'{e}zoutian}

\title{Vector spaces of linearizations for matrix polynomials: a bivariate polynomial approach}
\author{Yuji Nakatsukasa\thanks{Mathematical Institute, University of Oxford, Oxford, OX2 6GG,
UK. (\texttt{Yuji.Nakatsukasa@maths.ox.ac.uk}) Supported by EPSRC grant EP/I005293/1, and by JSPS through grant No.~26870149 and as a Postdoctoral Fellow for Research Abroad.
}
\and Vanni Noferini\thanks{Department of Mathematical Sciences, University of Essex, Wivenhoe Park, Colchester, CO4 3SQ, United Kingdom
                   (\texttt{vnofer@essex.ac.uk}). Supported by European Research Council Advanced Grant MATFUN (267526).}
\and Alex Townsend\thanks{Department of Mathematics, Cornell University, Ithaca, NY 14853. (\texttt{townsend@cornell.edu})}
}

\begin{document}
\maketitle

\begin{dedication}
In memory of Leiba Rodman
\end{dedication}

\begin{abstract}
We revisit the landmark paper [D. S. Mackey, N. Mackey, C. Mehl, and V. Mehrmann, {SIAM J. Matrix Anal. Appl.}, 28 (2006), pp.~971--1004] and, by viewing matrices as coefficients for bivariate polynomials, we provide concise proofs for key properties of linearizations for matrix polynomials. 
We also show that every pencil in the double ansatz space is intrinsically connected to a B\'{e}zout matrix, which we use to prove the eigenvalue exclusion theorem.
In addition our exposition allows for any polynomial basis and for any field.
The new viewpoint also leads to new results. We generalize the double ansatz space by exploiting its algebraic interpretation as a space of B\'{e}zout pencils to derive new linearizations with potential applications in the theory of structured matrix polynomials. Moreover, we analyze the conditioning of double ansatz space linearizations in the important practical case of a Chebyshev basis.
\end{abstract}



\section{Introduction}
The paper by Mackey, Mackey, Mehl, and 
 Mehrmann~\cite{Mackey_05_01} introduced three 
important vector spaces of pencils for matrix polynomials: $\mathbb{L}_1(P)$, $\mathbb{L}_2(P)$, 
and $\mathbb{DL}(P)$. In~\cite{Mackey_05_01} the spaces $\mathbb{L}_1(P)$ and $\mathbb{L}_2(P)$ generalize the companion forms of 
the first and second kind, respectively, and the \emph{double ansatz space} is the intersection, $\mathbb{DL}(P) = \mathbb{L}_1(P)\cap\mathbb{L}_2(P)$.  
These vector spaces provide a family of candidate generalized eigenvalue problems for computing the eigenvalues of a matrix polynomial, $P(\lam)$, giving a rich source of \emph{linearizations} for $P(\lambda)$: a classical approach for polynomial eigenvalue problems.

In this article we introduce new viewpoints for these vector spaces.  
We regard a block matrix as coefficients for a bivariate matrix polynomial
 (see Section~\ref{sec:bivariate}), and point out that every pencil in $\mathbb{DL}(P)$ is a (generalized) 
B\'{e}zout matrix
due to Lerer and Tismenetsky~\cite{Lerer_82_01} (see Section~\ref{sec:eigexclusion}).  These novel viewpoints allow us to obtain remarkably elegant proofs 
for many properties of $\mathbb{DL}(P)$ and the eigenvalue exclusion theorem, which previously required rather tedious derivations.
Furthermore, our exposition includes matrix polynomials expressed in any 
polynomial basis, such as the Chebyshev polynomial basis~\cite{Effenberger_11_01,kressner2013memory}. 
We develop a generalization of the double ansatz space (see Section~\ref{sec:beyondDL})
and also discuss extensions to generic algebraic fields, and conditioning analysis (see Section~\ref{sec:conditioning}). 

Let us recall some basic definitions in the theory of matrix polynomials. 
Let $P(\lam) = \sum_{i=0}^k P_i\phi_i(\lam)$ be 
a matrix polynomial expressed in a certain polynomial basis
$\left\{\phi_0,\ldots,\phi_k\right\}$, where $P_k\neq 0$, $P_i\in\mathbb{F}^{n\times n}$, and $\mathbb{F}$ is a field.
 Of particular interest is the case of a degree-graded basis, i.e.,\ 
$\left\{\phi_i\right\}$
 is a polynomial basis where $\phi_j$ is of exact degree $j$.
We assume throughout that  $P(\lambda)$ is regular, i.e.,  $\det P(\lambda)\not\equiv 0$, which ensures the finite eigenvalues of $P(\lam)$ are the roots of the scalar polynomial $\det(P(\lam))$. We note that if the elements of $P_i$ are in the field $\mathbb{F}$ then generally the finite eigenvalues exist in the algebraic closure of $\mathbb{F}$.

Given $X,Y \in \mathbb{F}^{nk\times nk}$ a matrix pencil $L(\lambda)=\lambda X + Y$ 
is a \emph{linearization} for $P(\lam)$ if there 
exist unimodular matrix polynomials $U(\lambda)$ and $V(\lambda)$, i.e., $\det U(\lambda),  \det V(\lambda)$ are nonzero elements of $\mathbb{F}$, such that $L(\lambda)=U(\lambda)\diag(P(\lambda),I_{n(k-1)})V(\lambda)$ and hence, $L(\lambda)$ shares its finite eigenvalues and their partial multiplicities with $P(\lam)$. 
If $P(\lam)$ has a singular leading coefficient, when expressed in a degree-graded basis, then it has an infinite eigenvalue and to preserve 
the partial multiplicities at infinity the matrix pencil $L(\lam)$ needs to be a \emph{strong linearization}, i.e., 
$L(\lam)$ is a linearization for $P(\lam)$ and $\lambda Y+ X$ a linearization for $\lambda^kP(1/\lambda)$. 

In the next section we recall the definitions of $\mathbb{L}_1(P)$, $\mathbb{L}_2(P)$, and $\mathbb{DL}(P)$ allowing for matrix polynomials expressed in any polynomial basis, extending the results in~\cite{Mackey_05_01} given for the monomial basis (such extension was also considered in~\cite{DeTeran_09_01}).
In Section~\ref{sec:bivariate} we consider the same space from a new viewpoint, based on bivariate matrix polynomials, and provide concise 
proofs for properties of $\mathbb{DL}(P)$. Section~\ref{sec:eigexclusion} shows that every 
pencil in $\mathbb{DL}(P)$ is a (generalized) B\'{e}zout matrix and gives an alternative proof of the 
eigenvalue exclusion theorem. In Section~\ref{sec:beyondDL} we generalize the double ansatz space to obtain a new family of linearizations, including new structured 
linearizations for structured matrix polynomials. 
Although these new linearizations are 
 mainly of theoretical interest they show how the new 
viewpoint can be used to derive novel results.
In Section~\ref{sec:conditioning} we analyze the conditioning of the eigenvalues of $\mathbb{DL}(P)$ pencils, and in Section~\ref{sec:construction} we describe a procedure to construct
block symmetric pencils in $\mathbb{DL}(P)$ and B\'{e}zout matrices.

\subsubsection*{Notation}  The expansion $P(\lam) = \sum_{i=0}^k P_i\phi_i(\lam)$ denotes a regular $n\times n$ matrix polynomial of degree $k$ expressed in a polynomial basis $\{\phi_i\}$.  The following vector $\left[\phi_{k-1}(\lam),\phi_{k-2}(\lam),\ldots,\phi_0(\lam)\right]^T$ is denoted by $\Lambda(\lam)$. The $n\times n$ identity matrix is denoted by $I_n$, which we also write as $I$ when the dimension is immediate from the context. 
The superscript $^\B$ represents blockwise transpose: if $X = (X_{ij})_{1\leq i,j\leq k}$, $X_{ij}\in\mathbb{F}^{n\times n}$, then $X^\B = (X_{ji})_{1\leq i,j \leq k}$. 

\section{Vector spaces and polynomial bases}\label{sec:vectorspaces}
Given a matrix polynomial $P(\lam)$ we can define a vector space, denoted by $\mathbb{L}_1(P)$, as~\cite[Def.\ 3.1]{Mackey_05_01}
\[
\mathbb{L}_1(P) = \left\{L(\lam) = \lam X + Y: X, Y \in\mathbb{F}^{nk\times nk}, \ L(\lam)\cdot (\Lambda(\lam)\otimes I_n ) = v\otimes P(\lam), v\in\mathbb{F}^k\right\},
\] 
where $\Lambda(\lam) = \left[\phi_{k-1}(\lam),\phi_{k-2}(\lam),\ldots,\phi_0(\lam)\right]^T$ and $\otimes$ is the matrix Kronecker product. An \emph{ansatz vector}
$v\in\mathbb{F}^k$ generates a family of pencils in $\mathbb{L}_1(P)$, which are generically linearizations for $P(\lam)$~\cite[Thm.~4.7]{Mackey_05_01}.
If $\left\{\phi_0,\ldots,\phi_k\right\}$ is an orthogonal basis, then the comrade form~\cite{Barnett_book,Specht_57_01} 
belongs to $\mathbb{L}_1(P)$ with $v=\left[1,0,\ldots,0\right]^T$.

The action of $L(\lam)=\lam X+Y \in \mathbb{L}_1(P)$ on $(\Lambda(\lam)\otimes I_n )$ can be 
characterized by the \emph{column shift sum} operator, denoted by $\css$~\cite[Lemma 3.4]{Mackey_05_01},
\[
L(\lam)\cdot (\Lambda(\lam)\otimes I_n ) = v\otimes P(\lam) \Longleftrightarrow  X\css Y = v\otimes \left[P_k, P_{k-1},\ldots,P_0\right]. 
\]
In the monomial basis $X\css Y$ can be paraphrased as ``insert a zero column on the right of $X$ and a zero column on the left of $Y$ then add them together'', i.e.,\
\[
X\css Y = \begin{bmatrix}X & \textbf{0}\end{bmatrix} +  \begin{bmatrix}\textbf{0} & Y\end{bmatrix},
\]
where $\textbf{0}\in\mathbb{F}^{nk\times n}$. More generally, given a polynomial basis we define the column shift sum operator as 
\begin{equation}\label{eq:Mcolshift}
X\css Y = X M  + \begin{bmatrix}\textbf{0} & Y\end{bmatrix},
\end{equation}
where $\textbf{0}\in\mathbb{F}^{nk\times n}$, and $M\in\mathbb{F}^{nk\times n(k+1)}$ has block elements $M_{pq} = m_{p,q}I_n$ for $1\leq p\leq k$ and $1\leq q\leq k+1$,
and $m_{p,q}$ is defined via the representation 
\begin{equation}  \label{eq:Mmat}
x\phi_{i-1} = \sum_{j=0}^k m_{k+1-i,k+1-j}\phi_{j}, \quad 1\leq i\leq k.   
\end{equation}
The matrix $M$ has a particularly nice form if the basis is degree-graded, since
the terms with $j>i$ in the above sum are zero. 
Then, the matrix $M$ in~\eqref{eq:Mcolshift} is given by
\begin{equation}\label{eq:defM}
M =\begin{bmatrix}M_{11}&M_{12}&\ldots&M_{1k}&M_{1,k+1}\cr \textbf{0} & M_{22} & \ddots & \ddots &M_{2,k+1} \cr \vdots&\ddots&\ddots&\ddots&\vdots\cr \textbf{0} & \ldots &\textbf{0} &M_{kk}&M_{k,k+1}\end{bmatrix},
\end{equation}
where $M_{pq} = m_{p,q}I_n$, $1\leq p\leq q\leq k+1$, $p\neq k+1$. 
Furthermore, for an orthogonal basis, a three-term recurrence is satisfied and in this case the matrix 
$M$ has only three nonzero block diagonals.
For example, if 
$P(\lambda) \in \mathbb{R}[\lam]^{n \times n}$ is expressed in the Chebyshev basis\footnote{Non-monomial bases can be of significant interest when working with numerical algorithms over some subfield of $\mathbb{C}$. 
For the sake of completeness, we note that in order to define the Chebyshev basis the field characteristics must be different than $2$.} $\left\{T_0(x),\ldots,T_k(x)\right\}$, where
$T_j(x)=\cos\left(j\cos^{-1}x\right)$ for $x\in [-1,1]$, we have
\begin{equation}\label{eq:Morth}
M = \begin{bmatrix}\tfrac{1}{2}I_n & 0 & \tfrac{1}{2}I_n\cr & \ddots &\ddots&\ddots \cr & & \tfrac{1}{2}I_n & 0 & \tfrac{1}{2}I_n\cr & & & I_n& 0\cr \end{bmatrix} \in \mathbb{R}^{nk\times n(k+1)}.
\end{equation}

The properties of the vector space $\mathbb{L}_2(P)$ are analogous to $\mathbb{L}_1(P)$~\cite{Higham_06_01}. If $L(\lam)=\lam X + Y$ is in $\mathbb{L}_2(P)$ then $L(\lam) = \lam X^\B + Y^\B$ belongs to $\mathbb{L}_1(P)$. 
This connection means that the action of $L(\lam)\in\mathbb{L}_2(P)$ is characterized by a \emph{row shift sum} operator, 
denoted by $\rss$, and defined as
\[
X \rss Y = \left(X^\B\css Y^\B\right)^\B = M^\B X + \begin{bmatrix}
\textbf{0}^T\\
Y
\end{bmatrix}.
\]
Here, we used the fact that $(X^\B M)^\B = M^\B X$, which follows from the structure $M_{pq} = m_{p,q}I_n$. 
\subsection{Extending the results to general polynomial bases}\label{sec:genpolybases}
Many of the derivations in~\cite{Mackey_05_01} are specifically for $P(\lam)$ expressed in a monomial basis, though 
the lemmas and theorems can be generalized to any polynomial basis. One approach to generalize~\cite{Mackey_05_01} is to use 
a change-of-basis matrix $S$ such that $\Lambda(\lam) = S[\lambda^{k-1},\dots,\lambda,1]^T$ and to define the mapping (see also~\cite{DeTeran_09_01})
\begin{equation}
  \label{eq:mapc}
\mathcal{C}\left(\hat{L}(\lambda)\right)=\hat{L}(\lambda) (S^{-1} \otimes I_n) = L(\lam),   
\end{equation}
where $\hat{L}(\lam)$ is a pencil in $\mathbb{L}_1(P)$ for the matrix polynomial $P(\lam)$ expressed in the monomial basis with the same ansatz vector as $L(\lambda)$.  
  In particular, the strong linearization theorem holds for any polynomial basis.
\begin{theorem}[Strong Linearization Theorem]\label{thm:linearization}
Let $P(\lam)$ be a regular matrix polynomial (expressed in any 
polynomial basis), and let $L(\lam)\in\mathbb{L}_1(P)$. Then 
the following statements are equivalent:
\begin{enumerate}
\item $L(\lam)$ is a linearization for $P(\lam)$,
\item $L(\lam)$ is a regular pencil, and
\item $L(\lam)$ is a strong linearization for $P(\lam)$.
\end{enumerate}
\end{theorem}
\begin{proof} It is a corollary of~\cite[Theorem 4.3]{Mackey_05_01}. In fact, the mapping $\mathcal{C}$ in~\eqref{eq:mapc} is a strict equivalence between $\mathbb{L}_1(P)$ expressed in the monomial basis and $\mathbb{L}_1(P)$ expressed in another polynomial basis. 
Therefore, $L(\lam)$ has one of the three properties if and only if $\hat{L}(\lam)$ also does, and the properties are equivalent for $\hat{L}(\lam)$ because they are 
equivalent for $L(\lam)$.
\end{proof}

This strict equivalence can be used to generalize many properties of $\mathbb{L}_1(P)$, $\mathbb{L}_2(P)$, and $\mathbb{DL}(P)$, including~\cite[Thm.~4.7]{Mackey_05_01}, which shows that a pencil from $\mathbb{L}_1(P)$ is generically a linearization; however, our approach based on bivariate polynomials allows for more concise derivations.

\section{Recasting to bivariate matrix polynomials}\label{sec:bivariate}
A block matrix $X\in\mathbb{F}^{nk\times nh}$ with $n\times n$ blocks can provide the 
coefficients 
in the basis $\{\phi_i\}$ 
for a bivariate matrix polynomial of degree 
$h-1$ in $x$ and $k-1$ in $y$. 
Let $\phi:\mathbb{F}^{nk\times nh}\rightarrow \mathbb{F}_{h-1}^{n\times n}[x]\times\mathbb{F}^{n\times n}_{k-1}[y]$ be the mapping defined by
\begin{equation}  \label{eq:phidef}
\phi \ : \ X = \begin{bmatrix}X_{11}  & \dots & X_{1h} \cr \vdots &\ddots &\vdots \cr X_{k1} & \ldots & X_{kh}\end{bmatrix}, \hbox{ }X_{ij}\in\mathbb{F}^{n\times n} \mapsto F(x,y) = \sum_{i=0}^{k-1}\sum_{j=0}^{h-1} X_{k-i,h-j}\phi_i(y)\phi_j(x).  
\end{equation}
Equivalently, we may define the map as follows:
\[ \phi\ : \ X = \begin{bmatrix}X_{11}  & \dots & X_{1h} \cr \vdots &\ddots &\vdots \cr X_{k1} & \ldots & X_{kh}\end{bmatrix} \mapsto F(x,y) = \begin{bmatrix}\phi_{k-1}(y) I &\cdots & \phi_0(y) I \end{bmatrix} X \begin{bmatrix}
\phi_{h-1}(x) I\\
\vdots\\
\phi_0(x) I
\end{bmatrix}.\]

Usually, and unless otherwise specified, we will apply  the map $\phi$ to square block matrices so that $h=k$.

We recall that a regular (matrix) polynomial $P(\lam)$ expressed in a degree-graded basis has an infinite eigenvalue if its 
leading matrix coefficient is singular. In order to correctly take care of infinite eigenvalues we 
write $P(\lambda)=\sum_{i=0}^g P_i \phi_i(\lambda)$, where 
the integer $g\geq k$ is called the \emph{grade}~\cite{Mackey_11_01}. If the grade of $P(\lam)$ is larger
than the degree then $P(\lam)$ has at least $n$ infinite eigenvalues. Usually, and unless stated otherwise, the grade is 
equal to the degree.

It is easy to show that the mapping $\phi$ is a bijection between $h \times k$ block matrices with $n\times n$ blocks
and $n\times n$ bivariate matrix polynomials of grade $h-1$ in $x$ and grade 
$k-1$ in $y$. Even more, $\phi$ is an isomorphism preserving the group additive structure. We omit the trivial proof.

Many matrix operations can be interpreted as functional operations via the above described duality between block matrices and their 
continuous analogues (see, for example,~\cite{townsend2013extension}). Bivariate matrix polynomials allow us to interpret many matrix operations in terms of functional operations. In many instances, existing 
proofs in the theory of linearizations of matrix polynomials can be  simplified, and throughout the paper we will often exploit this parallelism.
We summarize some computation rules in Table~\ref{tab:op}. We hope the table will be useful not only in this paper, but also for future work. All the rules are valid for any basis and for any field $\mathbb{F}$, except the last row that assumes $\mathbb{F}=\mathbb{C}$.

\begin{table}[htbp]
  \centering
  \caption{Correspondence between operations in the matrix $X$ and the bivariate polynomial viewpoints regarding $F(x,y)$.
The matrix $M$ is from~\eqref{eq:Mmat}, and $v,w\in\mathbb{F}^{k}$ are vectors.
}
  \label{tab:op}
\begin{tabular}{c|c}
Block matrix operation & Bivariate polynomial operation\\
\hline
Block matrix $X$ & Bivariate polynomial $F(x,y)$\\
 
$X \mapsto X M$ & $F(x,y) \mapsto F(x,y) x$\\
 
$X \mapsto M^{\mathcal{B}} X$ & $F(x,y) \mapsto y F(x,y)$ \\
 
$X (\Lambda(\lambda) \otimes I)$ & Evaluation at $x=\lambda$: $F(\lambda,y)$\\
 
$X (\Lambda(\lambda) \otimes v)$ & $F(\lambda,y) v$\\
 
$(\Lambda^T(\mu) \otimes w^T) X$ & 
 $w^TF(x,\mu)$\\
 
$(\Lambda^T(\mu) \otimes w^T) X (\Lambda(\lambda) \otimes v)$ & $w^TF(\lambda,\mu) v$\\
 
$X \mapsto X^\B$ & $F(x,y) \mapsto F(y,x)$\\

 $X \mapsto X^T$ & $F(x,y) \mapsto F^T(y,x)$\\
 
$X \mapsto X^*$ & $F(x,y) \mapsto F^*(y,x)$\\
\end{tabular}
\end{table}

Other computational rules exist when the basis has additional properties. We give some examples in Table~\ref{tab:op2}, in which
\begin{equation}  \label{eq:sigR}
\Sigma=
\small
\begin{bmatrix} \ddots & & & \\
& I & & \\
& & -I &\\
& & & I
\end{bmatrix}
\normalsize
, \quad R=\begin{bmatrix} & & I\\
& \iddots & \\
I & &
\end{bmatrix}  ,
\end{equation}
and we say that a polynomial basis is alternating if $\phi_{i}(x)$ is even (odd) when $i$ is even (odd).

\begin{table}[htbp]
  \centering
  \caption{
Correspondence when the polynomial basis is alternating or the monomial basis. 
}
  \label{tab:op2}
\begin{tabular}{c|c|c}

Type of basis & Block matrix operation & Bivariate polynomial operation\\
\hline
Alternating & $X \mapsto \Sigma X$ & $F(x,y) \mapsto F(x,-y)$\\

Alternating & $X \mapsto X \Sigma$ & $F(x,y) \mapsto F(-x,y)$\\

Monomials & $X \mapsto R X$ & $F(x,y) \mapsto y^{k-1}F(x,y^{-1})$\\

Monomials & $X \mapsto X R$ & $F(x,y) \mapsto x^{h-1}F(x^{-1},y)$\\

\end{tabular}
\end{table}

As seen in Table~\ref{tab:op}, the matrix $M$ in~\eqref{eq:Mcolshift} is such that the bivariate matrix polynomial corresponding to the coefficients $XM$ is 
$F(x,y)x$, i.e., $M$ applied on the right of $X$ represents multiplication of $F(x,y)$ by $x$. 
This gives an equivalent definition for the column shift sum operator: if the block matrices $X$ and $Y$ are the coefficients for 
$F(x,y)$ and $G(x,y)$ then the coefficients of $H(x,y)$ are $Z$, where 
\[
Z =  X\css Y,\qquad H(x,y) = F(x,y)x + G(x,y).
\]
This gives a characterization of the $\mathbb{L}_1(P)$ space from the bivariate polynomial viewpoint as pencils 
$L(\lam) = \lam X + Y$ 
such that with the mapping $\phi$, we have $F(x,y)x + G(x,y) = v(y)P(x)$. 
The ansatz vector is $v=[v_{k-1},\ldots,v_1,v_0]^T$, where $v(y) = \sum_{i=0}^{k-1}v_i\phi_i(y)$. 

Regarding the space $\mathbb{L}_2(P)$, the coefficient matrix $M^\B X$ corresponds to the bivariate matrix polynomial $yF(x,y)$, i.e., $M^\B$ applied on the left of $X$ represents multiplication of $F(x,y)$ by $y$. 
We have thus derived the following result (here and below, with 
a slight abuse of notation, we use $v$ to denote both the ansatz polynomial and the ansatz vector). 
\begin{lemma}\label{lem:L1L2}
For an $n\times n$ matrix polynomial $P(\lambda)$ of degree $k$, 
the space $\mathbb{L}_1(P)$ can be written as 
\[
\mathbb{L}_1(P) = \left\{L(\lam) = \lam X + Y: 
 F(x,y)x + G(x,y) = v(y)P(x), v\in\Pi_{k-1}(\mathbb{F})\right\},
\]  
where 
 $\Pi_{k-1}(\mathbb{F})$ is the space of polynomials in $\mathbb{F}[y]$ of degree at most $k-1$, and 
$F(x,y),G(x,y)$ are defined using $X,Y$ by the mapping~\eqref{eq:phidef}, and, writing $v(y) = \sum_i^{k-1}v_i\phi_i(y)$, $v=[v_{k-1},\ldots,v_1,v_0]^T$ is the ansatz vector. 
Similarly, 
\[
\mathbb{L}_2(P) = \left\{L(\lam) = \lam X + Y: 
yF(x,y) + G(x,y) = P(y)w(x), w\in\Pi_{k-1}(\mathbb{F})\right\}.
\]
\end{lemma}


The space $\mathbb{DL}(P)$ is the intersection of $\mathbb{L}_1(P)$ and $\mathbb{L}_2(P)$.
It is an important vector space because it contains block symmetric linearizations. 
By Lemma~\ref{lem:L1L2}, 
a pencil $L(\lam) = \lam X + Y$ belongs to $\mathbb{DL}(P)$ with ansatze $v(y)$ and $w(x)$ if the 
following $\mathbb{L}_1(P)$ and $\mathbb{L}_2(P)$ conditions are satisfied:
\begin{equation}\label{eq:DLrelations}
F(x,y)x + G(x,y) = v(y)P(x), \qquad yF(x,y) + G(x,y) = P(y)w(x).
\end{equation}
It appears that $v(y)$ and $w(x)$ could be chosen independently; however, if we substitute  $y=x$ into~\eqref{eq:DLrelations} we obtain the \emph{compatibility condition}
\[
v(x)P(x) = F(x,x)x + G(x,x) = xF(x,x) + G(x,x) = P(x)w(x)
\]
and hence, $v=w$ as elements of $\Pi_{k-1}(\mathbb{F})$ since $P(x)(v(x)-w(x))$ is the zero matrix.
This shows the double ansatz space is actually a single ansatz space; a fact that required two quite technical proofs in~\cite[Prop.~5.2, Thm.~5.3]{Mackey_05_01}.

The bivariate matrix polynomials $F(x,y)$ and $G(x,y)$ are uniquely defined by the ansatz $v(x)$ since they satisfy the explicit formulas
\begin{equation}\label{eq:F}
yF(x,y) - F(x,y)x = P(y)v(x) - v(y)P(x),
\end{equation}
\begin{equation}\label{eq:G}
yG(x,y)-G(x,y)x = yv(y)P(x) - P(y)v(x)x.
\end{equation}
In other words, there is an isomorphism between $\Pi_{k-1}(\mathbb{F})$ and $\mathbb{DL}(P)$. It also follows from~\eqref{eq:F} and~\eqref{eq:G} that $F(x,y) = F(y,x)$ and $G(x,y)=G(y,x)$. This shows 
that all the pencils in $\mathbb{DL}(P)$ are block symmetric.  Furthermore, if $F(x,y)$ and $G(x,y)$ are symmetric and satisfy
$F(x,y)x + G(x,y) = P(x)v(y)$ then we also have $F(y,x)x + G(y,x) = P(x)v(y)$, and by swapping $x$ and $y$ we obtain the $\mathbb{L}_2(P)$ condition,
$yF(x,y) + G(x,y) = P(y)v(x)$. This shows all block symmetric pencils in $\mathbb{L}_1(P)$ belong to $\mathbb{L}_2(P)$ and hence, 
also belong to $\mathbb{DL}(P)$. Thus, $\mathbb{DL}(P)$ is the space of block symmetric pencils in 
$\mathbb{L}_1(P)$~\cite[Thm.\ 3.4]{Higham_06_01}.  

\begin{remark}Although in this paper we do not consider singular matrix polynomials, we note that the analysis of this section still holds even if we drop the assumption that $P(x)$ is regular. We only need to assume $P(x) \not \equiv 0$ in our proof that $\mathbb{DL(P)}$ is in fact a single ansatz space, which is no loss of generality because $\mathbb{DL}(0)=\{0\}$.
\end{remark}

\section{Eigenvalue exclusion theorem and B\'{e}zoutians}\label{sec:eigexclusion}
The eigenvalue exclusion theorem~\cite[Thm.~6.7]{Mackey_05_01} shows
that if $L(\lam)\in\mathbb{DL}(P)$ with ansatz
$v\in\Pi_{k-1}(\mathbb{F})$, then $L(\lam)$ is a linearization for the
matrix polynomial $P(\lam)$ if and only if $v(\lam)I_n$ and $P(\lam)$ do not share an
eigenvalue.  This theorem is important because, generically, $v(\lam)I_n$ and $P(\lam)$ do not 
share eigenvalues and almost all choices for $v\in\Pi_{k-1}(\mathbb{F})$ correspond
to linearizations in $\mathbb{DL}(P)$ for $P(\lam)$.

\begin{theorem}[Eigenvalue Exclusion Theorem]\label{thm:eigexclusion}
  Suppose that $P(\lam)$ is a regular matrix polynomial of degree $k$ and $L(\lam)$ is in
  $\mathbb{DL}(P)$ with a nonzero ansatz polynomial $v(\lambda)$. Then,
  $L(\lam)$ is a linearization for $P(\lam)$ if and only if $v(\lam) I_n$ (with grade $k-1$) and
  $P(\lam)$ do not share an eigenvalue.
\end{theorem}

We note that the last statement also includes infinite eigenvalues. 
In the following we will observe that any
$\mathbb{DL}(P)$ pencil is a (generalized) B\'{e}zout matrix and expand on this theme.  
This observation tremendously simplifies the proof of Theorem~\ref{thm:eigexclusion} and the connection with 
the classical theory of B\'{e}zoutian (for the scalar case) and the Lerer--Tismenetsky B\'{e}zoutian (for the matrix case) 
allows us to further our understanding of 
 the $\mathbb{DL}(P)$ vector space, and leads to a new vector space of linearizations. 
 We first
recall the definition of a B\'{e}zout matrix and B\'{e}zoutian function for scalar polynomials (see~\cite[p.~277]{Bernstein_09_01} and~\cite[sec.~2.9]{Bini_94_01}).

\begin{definition}[B\'{e}zout matrix and B\'{e}zoutian function]\label{def:Bezout}
  Let $p_1(x)$ and $p_2(x)$ be scalar polynomials
  \[
  p_1(x) = \sum_{i=0}^k a_i\phi_i(x), \qquad p_2(x) =
  \sum_{i=0}^{k} c_i\phi_i(x),
  \]
\emph{(}$a_k$ and $c_k$ can be zero as we regard $p_1(x)$ and $p_2(x)$ as polynomials of grade $k$\emph{)}, then the B\'{e}zoutian 
function associated with $p_1(x)$ and $p_2(x)$ is 
 the bivariate function 

\[
\mathcal{B}(p_1,p_2)
=\sum_{i,j=1}^k
  b_{ij}\phi_{k-i}(y)\phi_{k-j}(x):=  \frac{p_1(y)p_2(x)-p_2(y)p_1(x)}{x-y}.
\]
The $k\times k$ B\'{e}zout matrix 
 associated to $p_1(x)$ and $p_2(x)$ is defined via the 
 coefficients of the B\'{e}zoutian function
 \begin{equation}
   \label{eq:bezmat}
  B(p_1,p_2) = \left(b_{ij}\right)_{1\leq i,j\leq k}.    
 \end{equation}
\end{definition}

Here are some standard properties of a B\'{e}zoutian function and B\'{e}zout matrix:

\begin{enumerate}
\item The B\'{e}zoutian function is skew-symmetric with respect to its polynomial arguments: $\mathcal{B}(p_1,p_2)=-\mathcal{B}(p_2,p_1)$. 
\item 
$\mathcal{B}(p_1,p_2)$ 
 is bilinear with respect to its polynomial
  arguments.
\item $B(p_1,p_2)$ is nonsingular if and only if $p_1$ and $p_2$
  have no common roots.
\item $B(p_1,p_2)$ is a symmetric matrix.
\end{enumerate}

Property 3 holds for polynomials whose coefficients lie in \emph{any} field $\mathbb{F}$, provided that the common roots are sought after in the algebraic closure of $\mathbb{F}$ and roots at infinity are included. Note in fact that the dimension of the B\'{e}zout matrix depends on the formal choice of the grade of $p_1$ and $p_2$. Unusual choices of the grade are not completely artificial: for example, they may arise when evaluating a bivariate polynomial along $x=x_0$ forming a univariate polynomial~\cite{Nakatsukasa_13_01}. 
Moreover, it is 
important to be aware that common roots at infinity make the B\'{e}zout matrix singular. 

\begin{example}
Consider the finite field $F_2=\{0,1\}$ and let $p_1=x^2$ and $p_2=x+1$, whose finite roots are counting multiplicity $\{0,0\}$ and $\{1\}$, respectively. The 
B\'{e}zout function is $x+y+xy$. If $p_1$ and $p_2$ are seen as grade $2$, the B\'{e}zout matrix (in the monomial basis) is $\begin{bmatrix}
1 & 1\\
1 & 0
\end{bmatrix}$, which is nonsingular and has a determinant of $1$. This is expected as $p_1$ and $p_2$ have no shared root. If $p_1$ and $p_2$ are seen as grade $3$ the B\'{e}zout matrix becomes $\begin{bmatrix}
0 & 0 & 0\\
0 & 1 & 1\\
0 &1 & 0
\end{bmatrix}$, whose kernel is spanned by $\begin{bmatrix}
1 & 0 & 0
\end{bmatrix}^T$. Note indeed that if the grade is $3$ then the roots are, respectively, $\{\infty,0,0\}$ and $\{\infty,\infty,1\}$, so $p_1$ and $p_2$ share a root at $\infty$.
\end{example}

To highlight the connection with the classic B\'{e}zout matrix we first 
consider scalar polynomials
 and show that the eigenvalue exclusion theorem immediately follows from the connection with B\'{e}zoutians.
\begin{proof}[Proof of Theorem~\ref{thm:eigexclusion} for $n=1$]
Let $p(\lam)$ be a scalar polynomial of degree (and grade) $k$ 
and $v(\lam)$ a scalar polynomial of degree at most $k-1$. 
  We first solve the relations in~\eqref{eq:F} and~\eqref{eq:G} to obtain
  \begin{equation}
    \label{eq:fgscalar}
  F(x,y) = \frac{p(y)v(x)-v(y)p(x)}{y-x}, \qquad G(x,y) =
  \frac{yv(y)p(x)-p(y)v(x)x}{y-x}    
  \end{equation}
  and thus, by Definition~\ref{def:Bezout},
  $F(x,y)=\mathcal{B}(v,p)$ and 
  $G(x,y)=\mathcal{B}(p,vx)$. 
Moreover, $\mathcal{B}$ is skew-symmetric and bilinear
  with respect to its polynomial arguments so
  \begin{equation}    \label{eq:Lbez}
  L(\lam) = \lam X + Y = \lam B(v,p) + B(p,xv) = -\lam B(p,v) +
  B(p,xv) =B(p,(x-\lam)v).    
  \end{equation}
 Since $B$ is a B\'{e}zout matrix, $\det(L(\lam))=\det(B(p,(x-\lam)v))\neq 0$ for some $\lam$  if and only if $p$ and $v$ do not share a root, which, by Theorem~\ref{thm:linearization}, is equivalent to $L(\lam)$ being a linearization. 
  \end{proof}
  
  An alternative (more algebraic) argument is to note that $p$ and $(x-\lam)v$ are polynomials in $x$ whose coefficients lie in the field of fractions $\mathbb{F}(\lam)$. Since $p$ has coefficients in the subfield $\mathbb{F} \subset \mathbb{F}(\lam)$, its roots lie in the algebraic closure of $\mathbb{F}$, denoted by $\overline{\mathbb{F}}$. The factorization $(x-\lam)v$ similarly reveals that this polynomial has one root at $\lam$, while all the others lie in $\overline{\mathbb{F}} \cup \{\infty\}$. Therefore, $p$ and $(x-\lam)v$ share a root in the closure of $\mathbb{F}(\lam)$ if and only if $p$ and $v$ share a root in $\overline{\mathbb{F}}$. Our proof of the eigenvalue exclusion theorem is purely algebraic and holds without any assumption on the field $\mathbb{F}$. However, as noted by Mehl~\cite{privcom}, if $\mathbb{F}$ is finite it could happen that no pencil in $\mathbb{DL}$ is a linearization, because there are only finitely many choices available for the ansatz 
polynomial $v$. Although this approach is extendable to any field, for simplicity of exposition we assume for the rest of this section that the underlying field is $\mathbb{C}$.
  
  A natural question at this point is whether this approach generalizes to the matrix case ($n>1$). An appropriate generalization of the scalar B\'{e}zout matrix should:  
  \begin{itemize}
  \item depend on two matrix polynomials $P^{(1)}$ and $P^{(2)}$;
  \item have nontrivial kernel if and only if $P^{(1)}$ and $P^{(2)}$ share an eigenvalue and the corresponding eigenvector (note that for scalar polynomials the only possible eigenvector is $1$, up to multiplicative scaling).
  \end{itemize}
  
  The following examples show that the most  straightforward ideas fail to satisfy the second property above.  
  \begin{example}\label{ex:naive}
Note that the most na\"{i}ve idea, i.e., 
$\frac{P^{(1)}(y)P^{(2)}(x)-P^{(2)}(y)P^{(1)}(x)}{x-y}$,
is generally not even a matrix polynomial (its elements are rational functions).

Almost as straightforward is the  generalization $\frac{P^{(1)}(y)P^{(2)}(x)-P^{(1)}(x)P^{(2)}(y)}{x-y}$, which is indeed a bivariate matrix polynomial. However, consider the associated B\'{e}zout block matrix. 
Let us check that it does not satisfy the property of being singular if and only if $P^{(1)}$ and $P^{(2)}$ have a shared eigenpair by providing two examples over the field $\mathbb{Q}$ and in the monomial basis. Consider first $P^{(1)}=\begin{bmatrix}
 x & 0\\
 0 & x-1
 \end{bmatrix}$ and $P^{(2)}=\begin{bmatrix}
 x-6 & -1\\
 12 & x+1
 \end{bmatrix}$. $P^{(1)}$ and $P^{(2)}$ have disjoint spectra. The corresponding B\'{e}zout matrix is $\begin{bmatrix}
 6 & 1\\
 -12 & -2
 \end{bmatrix}$, which is singular. Conversely, let
$P^{(1)}=\begin{bmatrix}
  x & 1\\
  0 & x
  \end{bmatrix}$ and $P^{(2)}=\begin{bmatrix}
  0 & x\\
  x & 1
  \end{bmatrix}$. Here, $P^{(1)}$ and $P^{(2)}$ share the eigenpair $\{0,\begin{bmatrix}
  1 & 0
  \end{bmatrix}^T\}$, but the corresponding B\'{e}zout matrix is $\begin{bmatrix}
  1 & 0\\
  0 & -1
  \end{bmatrix}$, which is nonsingular.
  \end{example}

  Fortunately, an extension of the B\'{e}zoutian to the matrix case was studied in the 1980s by Lerer, Tismenetsky, and others, see, e.g.,~\cite{Anderson_76_01, Lerer_82_01, LT2} and the references therein.  It turns out that it provides exactly the generalization that we need.
  
  \begin{definition}\label{def:bezmatpoly}
   For $n \times n$ regular matrix polynomials $P^{(1)}(x)$ and $P^{(2)}(x)$ of grade $k$, 
let $M^{(1)}(x)$ and $M^{(2)}(x)$ be regular matrix polynomials selected
so that $M^{(1)}(x)P^{(1)}(x)=M^{(2)}(x)P^{(2)}(x)$. 
Then, denoting the maximal degree of $M^{(1)}(x)$ and $M^{(2)}(x)$ by $\ell$,
the associated 
Lerer--Tismenetsky
B\'{e}zoutian function $\mathcal{B}_{M^{(2)},M^{(1)}}$ is 
   defined by~\cite{Anderson_76_01,Lerer_82_01} 
   \begin{equation}  \label{eq:bezoutmat}
      \mathcal{B}_{M^{(2)},M^{(1)}}(P^{(1)},P^{(2)})
=\!\!\sum_{i,j=1}^{\ell,k}\!\!
      B_{ij}\phi_{\ell-i}(y)\phi_{k-j}(x)\!:=\frac{M^{(2)}(y)P^{(2)}(x)-M^{(1)}(y)P^{(1)}(x)}{x-y}.
   \end{equation}
The $n\ell\times nk$ 
   B\'{e}zout block matrix is defined by $B_{M^{(2)},M^{(1)}}(P^{(1)},P^{(2)}) \!= \!\left(B_{ij}\right)_{1\leq i\leq \ell,1\leq j\leq k}$.
  \end{definition}

The Lerer--Tismenetsky B\'{e}zoutian function and the corresponding
B\'{e}zout block matrix are not unique as there are many possible choices of $M^{(1)}$ and $M^{(2)}$. Indeed, the matrix $B$ does not even need to be square. In the examples below we use monomials $\phi_i(x)=x^i$. 

\begin{example}
As in the first example in Example~\ref{ex:naive}, 
Let $P^{(1)}=\begin{bmatrix}
 x & 0\\
 0 & x-1
 \end{bmatrix}$ and $P^{(2)}=\begin{bmatrix}
 x-6 & -1\\
 12 & x+1
 \end{bmatrix}$ and select\footnote{$M^{(1)}$ and $M^{(2)}$ are of minimal degree as there are no $M^{(1)}$ and $M^{(2)}$ of degree $0$ or $1$ such that $M^{(1)}P^{(1)}=M^{(2)}P^{(2)}$ exists.} $M^{(1)}=\begin{bmatrix}
  x^2-3x+6 & x\\
  14x-12 & x^2+2x
  \end{bmatrix}$ and $M^{(2)}=\begin{bmatrix}
  x^2+3x & 2x\\
  2x & x^2
  \end{bmatrix}$. It can be verified that $M^{(1)}P^{(1)} = M^{(2)}P^{(2)}$. The associated 
 B\'{e}zout matrix is \[\begin{bmatrix}
  6 & 1\\
  -12 & -2\\
  -6 & 0\\
  -12 & 0
  \end{bmatrix}\] and has a trivial kernel.

Now consider again the second example in Example~\ref{ex:naive}, $P^{(1)}=\begin{bmatrix}
    x & 1\\
    0 & x
    \end{bmatrix}$, $P^{(2)}=\begin{bmatrix}
    0 & x\\
    x & 1
    \end{bmatrix}$, 
and select $M^{(1)} = P^{(1)}$ and $M^{(2)} = P^{(1)}F$, where $F = \begin{bmatrix}
    0 & 1\\
    1 & 0
    \end{bmatrix}$.  The 
B\'{e}zout matrix is
    $\begin{bmatrix}
    0 & 0\\
    0 & 0
   \end{bmatrix}$. Coherently with~\eqref{eq:kerbezout} below, its kernel has dimension $2$ because $P^{(1)}$ and $P^{(2)}$  only share the zero eigenvalue and the associated Jordan chain $\begin{bmatrix}
    1 & 0
    \end{bmatrix}^T,\begin{bmatrix}
        0 & -1
        \end{bmatrix}^T$.
\end{example}

When $P^{(1)}(x)$ and $P^{(2)}(x)$ commute, i.e., 
$P^{(2)}(x)P^{(1)}(x)=P^{(1)}(x)P^{(2)}(x)$, the natural choice of $M^{(1)}$ and $M^{(2)}$ is $M^{(1)}=P^{(2)}$ and $M^{(2)}=P^{(1)}$, and we write 
$\mathcal{B}(P^{(1)},P^{(2)}): = \mathcal{B}_{P^{(1)},P^{(2)}}(P^{(1)},P^{(2)})$.
 In this case the 
 B\'{e}zout matrix $B(P^{(1)},P^{(2)})$ (also dropping subscripts)
is square and 
of size $nk\times nk$.
Here are some important properties of the 
Lerer--Tismenetsky  B\'{e}zoutian function and the 
 B\'{e}zout matrix:

\begin{enumerate}
\item The Lerer--Tismenetsky B\'{e}zoutian function is skew-symmetric with respect to its arguments: $\mathcal{B}_{M^{(2)},M^{(1)}}(P^{(1)},P^{(2)})=-\mathcal{B}_{M^{(1)},M^{(2)}}(P^{(2)},P^{(1)})$,
\item 
$\mathcal{B}(P^{(1)},P^{(2)})$ 
 is bilinear with respect to its polynomial arguments. That is, 
$\mathcal{B}(a P^{(1)}+b P^{(2)},P^{(3)}) = a  \mathcal{B}(P^{(1)},P^{(3)})+b  \mathcal{B}(P^{(2)},P^{(3)})$
if $P^{(1)},P^{(2)}$ both commute with $P^{(3)}$,
\item  The kernel of the B\'{e}zout block matrix  is
\begin{equation}  \label{eq:kerbezout}
   \ker B_{M^{(2)},M^{(1)}}(P^{(1)},P^{(2)})
 =\mbox{Im}\ \!\!
   \begin{bmatrix}
     X_F\phi_{k-1}(T_F)\\
  \vdots\\
     X_F\phi_{0}(T_F)
   \end{bmatrix}
 \,\oplus \,\mbox{Im}\ \!\!
   \begin{bmatrix}
    X_\infty  \phi_{0}(T_\infty)\\
  \vdots\\
    X_\infty  \phi_{k-1}(T_\infty)
   \end{bmatrix}. 
\end{equation}
Here $(X_F,T_F)$, $(X_\infty,T_\infty)$ are the greatest common
   restrictions~\cite[Ch.~9]{Gohberg_09_01}
 of the finite and infinite Jordan pairs~\cite[Ch.~1, Ch.~7]{Gohberg_09_01} of $P^{(1)}(x)$ and $P^{(2)}(x)$. 
  The infinite Jordan pairs are defined regarding both polynomials as grade $k$.

Importantly, $   \ker  B_{M^{(2)},M^{(1)}}(P^{(1)},P^{(2)})$ in~\eqref{eq:kerbezout} does not depend on the choice of $M^{(1)}$ and $M^{(2)}$. This was proved (in the monomial basis) in~\cite[Thm.~1.1]{Lerer_82_01}.
Equation~\eqref{eq:kerbezout} holds for any polynomial basis: it can be obtained from 
that theorem via a congruence transformation involving the mapping 
$S^{-1} \otimes I_n$
 in~\eqref{eq:mapc}.
 
\item If for any $x$ and $y$ we have $P^{(1)}(y)P^{(2)}(x)=P^{(2)}(x)P^{(1)}(y)$, then $B(P^{(1)},P^{(2)})$ is a block symmetric matrix. Note that the hypothesis is stronger than $P^{(1)}(x)P^{(2)}(x)=P^{(2)}(x)P^{(1)}(x)$, but it is always satisfied when $P^{(2)}(x)=v(x) I$.
\end{enumerate}
   
   The following lemma shows that, as in the scalar case,  property 3 is the eigenvalue exclusion theorem in disguise.
   
   \begin{lemma}\label{lem:gcr}
      The greatest common restriction of the (finite and infinite) Jordan pairs of the regular matrix polynomials $P^{(1)}$ and $P^{(2)}$ is nonempty if and only if 
      $P^{(1)}$ and $P^{(2)}$ share both an eigenvalue and the corresponding
      eigenvector.
   \end{lemma}
   
   \begin{proof}
   Suppose that the two matrix polynomials have only finite eigenvalues. We denote by $(X_1,J_1)$ (resp., $(X_2,J_2)$) a Jordan pair of $P^{(1)}$ (resp., $P^{(2)}$).  Observe that a greatest common restriction is nonempty if and only if there exists at least one nonempty common restriction. 
First, assume there exist $v$ and $x_0$ such that $P^{(1)}(x_0)v=P^{(2)}(x_0)v=0$. 
Up to a similarity on the two Jordan pairs we have $X_1 S_1 e_1=X_2 S_2 e_1=v$, $J_1 S_1 e_1=S_1 e_1 x_0$, and $J_2 S_2 e_1=S_2 e_1 x_0$, where $S_1$ and $S_2$ are two similarity matrices. There is no loss of generality in assuming that, if necessary, one first applies the similarity transformation, see~\cite[p.~204]{Gohberg_09_01}.
This shows that $(v,x_0)$
 is a common restriction~\cite[p.~204, p.~235]{Gohberg_09_01} of the Jordan pairs of $P^{(1)}$ and $P^{(2)}$. 
Conversely, let $(X,J)$ be a common restriction with $J$ in Jordan form. We have the four equations $X_1 S_1 = X$,  $X= X_2 S_2$, $J_1 S_1 = S_1 J$, and $J_2 S_2 = S_2 J$ for some full column rank matrices $S_1$ and $S_2$. Letting $v:=X e_1, x_0:=e_1^T J e_1$, it is easy to check that $(v,x_0)$ is also a common restriction, and that $X_1 S_1 e_1= v = X_2 S_2 e_1$, $J_1 S_1 e_1 = S_1 e_1 x_0$, and $J_2 S_2 e_1 = S_2 e_1 x_0$. From~\cite[eq.~1.64]{Gohberg_09_01}\footnote{Although strictly speaking~\cite[eq.~1.64]{Gohberg_09_01} is for a monic matrix polynomial, it is extendable in a straightforward way to a regular matrix polynomial (see also~\cite[Ch.~7]{Gohberg_09_01}).}, it follows that $P^{(1)}(x_0)v=P^{(2)}(x_0)v=0$.

The assumption that all the eigenvalues are finite can be easily removed (although complicating the notation appears unavoidable). 
In the argument above replace every Jordan pair $(X,J)$ with a \emph{decomposable pair}~\cite[pp.~188--191]{Gohberg_09_01} of the form $\left[ X_F, X_\infty \right]$ and $ J_F \oplus J_\infty$, where $(X_F,J_F)$ is a finite Jordan pair and $(X_\infty, J_\infty)$ is an infinite Jordan pair~\cite[Ch.~7]{Gohberg_09_01}. As the argument is essentially the same we omit the details.
   \end{proof}
   
   The importance of the connection with B\'{e}zout theory is now clear.
   The proof of the eigenvalue exclusion theorem in the matrix polynomial case becomes immediate. Before giving the proof for Theorem~\ref{thm:eigexclusion} for $n>1$, we state the analogue of~\eqref{eq:Lbez} for matrix polynomials. Here and below, we use the notation $\mathbb{DL}(P,v)$  to denote the unique pencil in $\mathbb{DL}(P)$ with ansatz $v$.
   \begin{lemma}\label{lem:Lbzefunc}
$\mathbb{DL}(P,v)$ for a matrix polynomial $P(\lambda)$ with ansatz $v$ 
is a matrix pencil 
 that 
can be written as 
\begin{equation}  \label{eq:Lbezfunc}
\mathbb{DL}(P,v)
=  B(P,(x-\lambda)vI) 
=\lambda  B(vI,P)+B(P,xvI), 
\end{equation}
where $B$ is the B\'{e}zout matrix as in~\eqref{eq:bezmat}. 
   \end{lemma}
   \begin{proof}
As in~\eqref{eq:fgscalar} for the scalar case, 
  we solve~\eqref{eq:F} and~\eqref{eq:G} for $F(x,y)$ and $G(x,y)$ to obtain 
\[
  F(x,y) = \frac{v(y)P(x)-P(y)v(x)}{x-y}, \qquad G(x,y) =
  \frac{P(y)v(x)x-yv(y)P(x)}{x-y}    .
\]
Let $P^{(1)}=P(x)$ and $P^{(2)}=(x-\lambda)v(x) I_n$
in~\eqref{eq:bezoutmat}. Then, $P^{(1)}$ and $P^{(2)}$ commute
for all $x$, so we take $M^{(1)}=P^{(2)}$ and $M^{(2)}=P^{(1)}$ and obtain
 \begin{align*}
   \mathcal{B}(P(x),(x-\lambda)v(x)I_n)=
&\frac{P(y)(x-\lambda)v(x) -(y-\lambda)v(y)P(x)}{x-y}\\
 =& \lambda F(x,y)+G(x,y)\\
 =& \sum_{i,j=1}^k   B_{ij}\phi_{k-i}(y)\phi_{k-j}(x). 
 \end{align*}
This gives the $nk\times nk$ B\'{e}zout block matrix $B(P,(x-\lambda)vI)=(B_{ij})_{1\leq i,j\leq k}$.     
   \end{proof}

We are now ready to prove Theorem~\ref{thm:eigexclusion}. 
Recall that the claim is that $L(\lam)\in \mathbb{DL}(P) $ with ansatz $v(\lam)$ is a linearization for $P(\lam)$ if and only if $v(\lam) I_n$ and
  $P(\lam)$ have no shared eigenvalue.

\begin{proof}[Proof of Theorem~\ref{thm:eigexclusion} for $n>1$]
  If $v I_n$ and $P$ share a finite eigenvalue $\lam_0$ and 
$P(\lam_0)w=0$ for a nonzero $w$, then $(\lam_0-\lambda)v(\lam_0)w=0$
 for all
  $\lambda$. Hence, 
by~\eqref{eq:kerbezout} and Lemmas~\ref{lem:gcr} and~\ref{lem:Lbzefunc}, 
$L(\lambda)=B(P,(x-\lambda)vI)$ is singular for all $\lambda$. An analogous argument holds for a shared infinite eigenvalue. 
  Conversely, suppose $v(\lambda)I_n$ and $P(\lambda)$ have no common eigenvalues. 
 If $\lambda_0$ is an eigenvalue of $P$ then $(\lambda_0-\lambda)v(\lambda_0)I$ is nonsingular unless $\lambda = \lambda_0$. Thus, 
again using~\eqref{eq:kerbezout} and Lemma~\ref{lem:gcr}, 
if $\lam$ is not an eigenvalue for $P$ then 
the common restriction is empty, which means $L(\lambda)$ is nonsingular. In other words, $L(\lambda)$ is regular and a linearization by Theorem~\ref{thm:linearization}. 
\end{proof}

\section{Barnett's theorem and ``beyond $\mathbb{DL}$'' linearization space}\label{sec:beyondDL}
Thus far, we have introduced a new viewpoint for the $\mathbb{DL}$ linearization space and demonstrated that the viewpoint provides deeper understanding of the $\mathbb{DL}$ space, and often helps simplify proofs of known results. 
We now turn to new aspects and results that can be obtained from this viewpoint, namely, we define and study a new vector space of potential linearizations for matrix polynomials that includes $\mathbb{DL}$ as a subspace and extends many of its properties.
 
In this section we work for simplicity in the monomial basis, and we assume that the matrix polynomial $P(x) = \sum_{i=0}^k P_i x^i$ has an invertible leading coefficient $P_k$. Given a ring $R$, a \emph{left ideal} $L$ is a subset of $R$ such that $(L,+)$ is a subgroup of $(R,+)$ and $r \ell \in L$ for any $\ell \in L$ and $r \in R$~\cite[Ch.~1]{hazewinkel2004algebras}. A \emph{right ideal} is defined analogously.

Given a matrix polynomial $P(x)$ over some field $\mathbb{F}$ the set $L_P=\{ Q(x) \in \mathbb{F}^{n \times n}[x] \ | \ Q(x)=A(x)P(x), \ A(x)\in \mathbb{F}^{n \times n}[x]\}$ is a left ideal of the ring $\mathbb{F}^{n \times n}[x]$. Similarly, $R_P=\{ Q(x) \in \mathbb{F}^{n \times n}[x] \ | \ Q(x)=P(x)A(x), \ A(x)\in \mathbb{F}^{n \times n}[x] \}$ is a right ideal of $\mathbb{F}^{n \times n}[x]$.

A matrix polynomial of grade $k-1$ can be represented as $G(x)=\Gamma \Phi(x)$, where $\Gamma = [\Gamma_{k-1},\Gamma_{k-2},\ldots,\Gamma_0]\in \mathbb{F}^{n \times nk}$ are its coefficient matrices when expressed in the monomial basis and $\Phi(x)=\begin{bmatrix}
x^{k-1} I,\ldots,x I,I
\end{bmatrix}^\mathcal{B}$. 
Let $C_P^{(1)}$ be the first companion matrix\footnote{Some authors define the first companion matrix 
with minor differences in the choice of signs. Here, we make our choice for simplicity of what follows. 
For other polynomial bases the matrix should be replaced accordingly~\cite{Barnett_book}.
} of $P(x)$:
\begin{equation}
  \label{eq:comp1}
 C_P^{(1)} = \begin{bmatrix} 
-P_k^{-1} P_{k-1} & 
-P_{k}^{-1}P_{k-2} & \cdots & -P_{k}^{-1}P_1 & -P_{k}^{-1}P_0\\
I & & & & \\
& I & & & \\
& & \ddots & &\\
& & & I & 0
\end{bmatrix}.  
\end{equation}
A key observation is that
 the action of $C_P^{(1)}$ on $\Phi$ is that of the multiplication-by-$x I$ operator in the quotient module $\mathbb{F}^{n \times n}[x]/L_P$: 
\[ \begin{bmatrix}
x^{k-1} I\\
x^{k-2} I\\
\vdots\\
x I\\
I
\end{bmatrix} xI \equiv C_P^{(1)} \begin{bmatrix}
x^{k-1} I\\
x^{k-2} I\\
\vdots\\
x I\\
I
\end{bmatrix} = 
 \begin{bmatrix}
x^{k-1} I\\
x^{k-2} I\\
\vdots\\
x I\\
I
\end{bmatrix} xI + \begin{bmatrix}
- P_k^{-1}P(x)\\
0\\
\vdots\\
0\\
0
\end{bmatrix}, \ \ \ \ \ \begin{bmatrix}
- P_k^{-1}P(x)\\
0\\
\vdots\\
0\\
0
\end{bmatrix} \in L_P^{k}.\]

Multiplying by the coefficients $\Gamma$, we can identify the map $\Gamma \mapsto \Gamma C_P^{(1)}$ with the map $G(x) \mapsto G(x)x$ in $\mathbb{F}^{n \times n}[x]/L_P$. That is, we can write $\Gamma C_P^{(1)} \Phi = x G(x) + Q(x)$ for some $Q(x) \in L_P$. More precisely, we have $Q(x)=-\Gamma_{k-1}P_k^{-1}P(x)$. Applying the previous observation to each block row of a block matrix, we see that, if we map a block matrix $X \in  \mathbb{F}^{kn \times kn}$ to the corresponding bivariate matrix polynomial $\phi(X) \in  \mathbb{F}^{n \times n}[x,y]$ (recall the definition of $\phi$ in~\eqref{eq:phidef}), we can identify the map $X \mapsto X C_P^{(1)}$ with the map $\phi(X) \mapsto \phi(X) x$ in the quotient space $\mathbb{F}[y]^{n \times n}[x]/L_{P(x)} = \mathbb{F}^{n \times n}[x,y]/L_{P(x)}$. 

The next theorem shows that, when working with the quotient space modulo $L_P$ or $R_P$, one can find a unique matrix polynomial of low grade in each equivalence class.

\begin{theorem}\label{quotient}
Let $P(x)=\sum_{i=0}^k P_i x^i \in \mathbb{F}^{n \times n}[x]$ be a matrix polynomial of degree $k$ such that $P_k$ is invertible, and let $V(x) \in \mathbb{F}^{n \times n}[x]$ be any matrix polynomial. Then, there exists a unique $Q(x)$ of grade $k-1$ such that $Q(x) \equiv V(x)$ in the quotient module $\mathbb{F}^{n \times n}[x]/L_P$, i.e., there exists a unique $A(x) \in \mathbb{F}^{n \times n}[x]$ such that $V(x)= A(x) P(x) + Q(x)$ with $Q(x)$ of grade $k-1$.

Moreover, there exists a unique $S(x)$ of grade $k-1$ such that $S(x) \equiv V(x)$ in the quotient module $\mathbb{F}^{n \times n}[x]/R_P$, i.e., there exists a unique $B(x) \in \mathbb{F}^{n \times n}[x]$ such that $V(x) = P(x) B(x) + S(x)$ with $S(x)$ of grade $k-1$.
\end{theorem}
\begin{proof}
If $\deg V(x) < k$, then take $Q(x)=V(x)$ and $A(x)=0$. If $\deg V \geq k$, then our task
 is to find $A(x) = \sum_{i=0}^{\deg V-k} A_ix^i$ with $A_{\deg V- k} \neq 0$ such that, for $M(x)=A(x)P(x)=\sum_{i=0}^{{\deg V}} M_i x^i$, we have $M_i = \sum_{j+\ell=i} (A_j P_{\ell})= V_i$ for $k \leq i \leq \deg V$. 
This is equivalent to solving the following block matrix equation:
\begin{equation}\label{eq:whatisA} 
\begin{bmatrix} A_{\deg A} & \cdots  & A_0 \end{bmatrix} \begin{bmatrix} P_k & P_{k-1} & P_{k-2} & \cdots  & \\
 & P_k & P_{k-1} & P_{k-2}  & \cdots  \\ 
& & \ddots & \ddots  & \\
& &  & P_k & P_{k-1}\\
& & & & P_k
\end{bmatrix} = \begin{bmatrix} V_{\deg V} & \cdots & V_k \end{bmatrix},
\end{equation}
which shows explicitly that $A(x)$ exists and is unique. This also implies that $Q(x)=V(x) - A(x) P(x)$ exists and is unique. 
An analogous argument proves the existence and uniqueness of $B(x)$ and $S(x)$ such that $V(x) = P(x) B(x) + S(x)$.
\end{proof}

Thanks to the connection between $\mathbb{DL}$ and the B\'{e}zoutian, we find that~\cite[Theorem 4.1]{Higham_06_01} is a generalization  of Barnett's theorem to the matrix case. The proof that we 
give below 
is a generalization of that 
found 
 in~\cite{Heinig_1984} for the scalar case. 
It is 
 another 
 example 
  where the algebraic
interpretation and the connection with B\'{e}zoutians simplify proofs (compare with the argument in~\cite{Higham_06_01}).

\begin{theorem}[Barnett's theorem for matrix polynomials]\label{thm:Barnett}
Let $P(x)$ be a matrix polynomial of degree $k$ with nonsingular leading coefficient
and $v(x)$ a scalar polynomial of grade $k-1$. We have $\mathbb{DL}(P,v) = \mathbb{DL}(P,1) v(C_P^{(1)})$, where $C_P^{(1)}$ is the first companion matrix of $P(x)$.
\end{theorem}
\begin{proof}
It is easy to verify that the following recurrence formula holds:
\[ 
\begin{aligned} 
\frac{P(y) x^j (x-\lambda) - y^j (y-\lambda) P(x)}{x-y} &= \frac{P(y) x^{j-1}(x-\lambda) -y^{j-1}(y-\lambda)P(x)}{x-y} x \\
&\qquad\qquad\qquad\qquad\qquad\qquad+ y^{j-1} (y-\lambda) P(x).  
\end{aligned} 
\]
Hence, we have $\mathcal{B}(P,(x-\lambda) x^j I) \equiv \mathcal{B}(P,(x-\lambda) x^{j-1} I) x$ where the equivalence is in the quotient space $\mathbb{F}[y,\lambda]^{n \times n}[x]/L_{P(x)}$. On the other hand, as we argued above, the operator of multiplication-by-$x$ in the quotient space is represented, in the monomial basis, by right-multiplication times the matrix $C_P^{(1)}$, while, again in the monomial basis, the B\'{e}zoutian $\mathcal{B}(P,(x-\lambda) x^j I)$ (resp.~$\mathcal{B}(P,(x-\lambda) x^{j-1} I)$) is represented by the pencil $\mathbb{DL}(P,x^j)$ (resp.~$\mathbb{DL}(P,x^{j-1})$). This observation suffices to prove by induction the theorem when $v(C_P^{(1)})$ is a monomial of the form $(C_P^{(1)})^j$ for $0 \leq j \leq k-1$. The case of a generic $v(C_P^{(1)})$ follows by linearity of the B\'{e}zoutian.
\end{proof}

An analogous interpretation as a multiplication operator holds for the second companion matrix:
\begin{equation}
  \label{eq:comp2}
 C_P^{(2)} = \begin{bmatrix} 
- P_{k-1} P_k^{-1}& I & & & \\
-P_{k-2}P_{k}^{-1} & & I & & \\
\vdots &  & &\ddots  & \\
 -P_1P_{k}^{-1} & & & & I\\
-P_0P_{k}^{-1} & & &  & 
\end{bmatrix}.  
\end{equation}
Indeed, $C_P^{(2)}$ represents multiplication by $y$ modulo $R_P$, the right ideal generated by $P(y)$, i.e., if $X \in \mathbb{F}^{kn \times kn}$ is a block matrix we can identify the map $X \mapsto C_P^{(2)} X$ with the map $\phi(X) \mapsto y\phi(X)$ in  $\mathbb{F}[x]^{n \times n}[y]/R_{P(y)} = \mathbb{F}^{n \times n}[x,y]/R_{P(y)}$. 
A dual version of Barnett's theorem holds for the second companion matrix. Indeed, one has $\mathbb{DL}(P,v(x)) = v(C_P^{(2)}) \mathbb{DL}(P,1)$. The proof is analogous to the one for Theorem~\ref{thm:Barnett} and is omitted.

As soon as we interpret the two companion matrices in this way, we are implicitly defining a map $\psi$ from block matrices to bivariate polynomials modulo $L_{P(x)}$ and $R_{P(y)}$. More formally, let $S(x,y) \in \mathbb{F}^{n \times n}[x,y]$, and consider the equivalence class 
\begin{align}\label{eq:eqclass}
[S(x,y)] : = \{ &T(x,y) \in \mathbb{F}^{n \times n}[x,y] : T(x,y) = S(x,y) + L(x,y)P(x) \nonumber\\
&+ P(y) R(x,y) \ \ \mathrm{for} \ \mathrm{some} \ L(x,y), R(x,y) \in \mathbb{F}^{n \times n}[x,y] \}.
\end{align}
Moreover, if $[S_1(x,y)]=[S_2(x,y)]$, we write $S_1(x,y) \equiv S_2(x,y)$.
 Then, for any block matrix $X \in \mathbb{F}^{nk \times nk}$ we define $\psi(X) = [\phi(X)]$, where $\phi$ is again the map defined in~\eqref{eq:phidef}. In this setting, $\psi(X)$ is seen as an equivalence class, and we may summarize our analysis on the two companion matrices with the following equations:
\begin{equation}\label{eq:psicompanion}
 \psi(X C_P^{(1)}) = [\phi(X)x], \qquad \psi( C_P^{(2)} X) = [y \phi(X)], \qquad \mathrm{for} \ \mathrm{all} \ \ X \in \mathbb{F}^{kn \times kn}.
\end{equation}
Note that, by linearity,~\eqref{eq:psicompanion} imply in turn that  for any polynomials $v(y)$ and $w(x)$ we have:
\begin{equation}\label{eq:psicompanion2}
 \psi(X w(C_P^{(1)})) = [\phi(X)w(x)], \qquad \psi( v(C_P^{(2)}) X) = [v(y) \phi(X)], \qquad \mathrm{for} \ \mathrm{all} \ \ X \in \mathbb{F}^{kn \times kn}.
\end{equation}
 However, in the equivalence class $\psi(X)$ there exists a unique bivariate polynomial having grade equal to $\deg P-1$ separately in both $x$ and $y$, as we now prove in Theorem~\ref{biquotient}. (Clearly, this unique bivariate polynomial must be precisely $\phi(X)$, as the latter has indeed grade $\deg P-1$ separately in both $x$ and $y$.)
Theorem~\ref{biquotient} gives the appropriate matrix polynomial analogue to Euclidean polynomial division applied both in $x$ and $y$.
\begin{theorem}\label{biquotient}
Let $P(z)=\sum_{i=0}^k P_i z^i \in \mathbb{F}^{n \times n}[z]$ be a matrix polynomial with $P_k$ invertible, and let $F(x,y) = \sum_{i=0}^{k_1} \sum_{j=0}^{k_2} F_{ij} x^i y^j \in \mathbb{F}^{n \times n}[x,y]$ be a bivariate matrix polynomial. Then there is a unique decomposition $F(x,y) = Q(x,y) + A(x,y) P(x) + P(y) B(x,y) + P(y) C(x,y) P(x)$ such that 
\begin{enumerate}
\item[(i)] $Q(x,y)$, $A(x,y)$, $B(x,y)$ and $C(x,y)$ are all bivariate matrix polynomials,
\item[(ii)] $Q(x,y)$ has degree at most $k-1$ separately in $x$ and $y$, 
\item[(iii)] $A(x,y)$ has degree at most $k-1$ in $y$, and
\item[(iv)] $B(x,y)$ has degree at most $k-1$ in $x$.
\end{enumerate}
Moreover, $Q(x,y)$ is determined uniquely by $P(z)$ and $F(x,y)$.
\end{theorem}
\begin{proof} 
Let us first apply Theorem~\ref{quotient} taking $\mathbb{F}(y)$ as the base field. Then there exist unique $A(x,y)$ and $Q_1(x,y)$ such that $F(x,y) = A_1(x,y) P(x) + Q_1(x,y)$, where $A_1(x,y)$ and $Q_1(x,y)$ are polynomials in $x$. Furthermore, $\deg_x Q_1(x,y) \leq k-1$. A priori, the entries of $A_1(x,y)$ and $Q_1(x,y)$ could be rational functions in $y$. However, a careful analysis of~\eqref{eq:whatisA} shows that the coefficients of $A_1(x,y)=\sum_i A_{1,i}(y) x^i$ can be obtained by solving a block linear system $M w = v$, say, where $v$ depends polynomially in $y$ whereas $M$ is \emph{constant} in $y$. Hence, $A_1(x,y)$, and a fortiori $Q_1(x,y) = F(x,y) - A_1(x,y)P(x)$, are also polynomials in $y$. At this point we can apply Theorem~\ref{quotient} again to write (uniquely) $Q_1(x,y) = Q(x,y) + P(y) B(x,y)$ and $A_1(x,y)=A(x,y) + P(y) C(x,y)$, where $\deg_y Q(x,y)$ and $\deg_y A(x,y)$ are both at most $k-1$. Moreover, comparing again with~\eqref{eq:whatisA}, it is easy to check that it must also hold $\deg_x Q(x,y) \leq k-1$ and $\deg_x B(x,y) \leq k-1$.  Hence, $F(x,y) = Q(x,y) + A(x,y) P(x) + P(y) B(x,y) + P(y) C(x,y) P(x)$ is the sought decomposition.
\end{proof}

The next example illustrates the concepts just introduced.
\begin{example}\label{examplequotients}
Let $P(x)=I x^2 + P_1 x+ P_0$ and consider the block matrix $X=\begin{bmatrix} A & B\\
C & D
\end{bmatrix}$. We have $\phi(X) = A x y + B y + C x + D$. Let $Y=C_P^{(2)} X C_P^{(1)}$. Then, using~\eqref{eq:psicompanion}, we know that 
$\psi(Y) = [y \phi(X) x] =[A x^2 y^2 + B x y^2 + C x^2 y + D x y].$ In particular, observing that $I x^2 \equiv - P_1 x - P_0$ and that $I y^2 \equiv -P_1 y - P_0$, we have 
\begin{align*}
A& x^2 y^2 + B x y^2 + C x^2 y + D x y \equiv (-P_1 y -P_0)(A x^2 + Bx) + C x^2 y + D x y \\
&\equiv (-P_1 A y - P_0 A + C y)(-P_1 x - P_0) + (D - P_1 B) x y - P_0 B x \\
&= (P_1 A P_1 + D - P_1 B - C P_1) x y + (P_1 A P_0 - C P_0) y + (P_0 A P_1 - P_0 B) x + P_0 A P_0\\
& = \phi(Y)  , 
\end{align*}
 as by Theorem~\ref{biquotient} in the equivalence class $\psi(Y)$ there exists a unique element of grade $\deg P-1$ separately in $x$ and $y$, and by the definition of the mapping $\phi$ this unique element must be equal to $\phi(Y)$.

Equivalently, we could have taken quotients directly on the bases. The argument is that $\begin{bmatrix} y^2 I & y I\end{bmatrix}X\begin{bmatrix} x^2 I\\
x I
\end{bmatrix} \equiv \begin{bmatrix} -P_1 y -P_0 & y I\end{bmatrix}X\begin{bmatrix}
-P_1 x -P_0\\
x I
\end{bmatrix} = \psi(Y)$, and leads to the same result. 

A third way of computing $Y = C_P^{(2)} X C_P^{(1)}$ is to formally apply the linear algebraic definition of matrix multiplication, and then apply the mapping $\phi$ as in~\eqref{eq:phidef} (forgetting about quotient spaces). 

One remarkable consequence of Theorem~\ref{biquotient} is that these three approaches are all equivalent.
Note that the same remarks, using~\eqref{eq:psicompanion2}, apply to any block matrix of the form $\psi(v(C_P^{(2)}) X w(C_P^{(1)}))$, for any pair of polynomials $v(y)$ and $w(x)$.

For this example, we have taken a monic $P(x)$ for simplicity. If its leading coefficient $P_k$ is not the identity matrix, but still is nonsingular, then the explicit formulas become more complicated and involve $P_k^{-1}$.
\end{example}



\subsection{Beyond $\mathbf{\mathbb{DL}}$ space}
The key message in 
 Theorem~\ref{thm:Barnett} is that one can start with the pencil in $\mathbb{DL}$ associated with ansatz polynomial $v=1$ and repeatedly multiply the first companion matrix $C_P^{(1)}$ on the right, to obtain all the pencils in the ``canonical basis'' of $\mathbb{DL}$~\cite{Mackey_05_01}. In the scalar case ($n=1$) there is a bijection between pencils in $\mathbb{DL}$ and polynomials in $C_P^{(1)}$. However, the situation is quite different when $n>1$, as the vector space of polynomials in $C_P^{(1)}$ can have dimension up to $kn$, depending on the Jordan structure of $P(x)$. 
\begin{remark}For 
some matrix polynomials $P(x)$, the dimension of the vector space of polynomials in $C_P^{(1)}$ can be much lower than $nk$, although generically this upper bound is achieved. An extreme example is $P(x)=p(x)I$ for some scalar $p(x)$, as in this case the dimension achieves the lowest possible bound, which is $k$.\end{remark} 

\subsubsection{Definition of $\mathbb{BDL}(P,v)$}
In light of the above discussion, it makes sense to investigate the pencils of the form $v(C_P^{(2)})\mathbb{DL}(P,1)=\mathbb{DL}(P,1) v(C_P^{(1)})$ for $\deg v > k-1$, because for a generic $P$ they do not belong to $\mathbb{DL}$.  We refer to the space of such pencils as the ``beyond $\mathbb{DL}$'' space of potential linearizations and write 
\begin{equation}  \label{eq:BDLdef}
\mathbb{DL}(P,1) v(C_P^{(1)})=:\mathbb{BDL}(P,v).   
\end{equation}
Note that $\mathbb{DL}$ is now seen as a subspace of $\mathbb{BDL}$: if $\deg v \leq k-1$, then $\mathbb{BDL}(P,v)=\mathbb{DL}(P,v)$. 

An important fact is that, even if the degree of the polynomial $v(x)$ is larger than $k-1$, it still holds that $\mathbb{BDL}(P,v)=v(C_P^{(2)})\mathbb{DL}(P,1)=\mathbb{DL}(P,1) v(C_P^{(1)})$. When $\deg v \leq k-1$, i.e., for pencils in $\mathbb{DL}$, this is a consequence of the two equivalent versions of Barnett's theorem. We now prove this more generally. 
\begin{theorem}
Let $P(x)$ be a matrix polynomial of degree $k$ with nonsingular leading coefficient.
For any polynomial $v(x)$ we have 
\[
v(C_P^{(2)})\mathbb{DL}(P,1)=\mathbb{DL}(P,1) v(C_P^{(1)}),
\]
where $C_P^{(1)},C_P^{(2)}$ are the companion matrices as in~\eqref{eq:comp1},~\eqref{eq:comp2} and $\mathbb{DL}(P,1)$ is the pencil as in~\eqref{eq:Lbezfunc}. 
\end{theorem}
\begin{proof}
Since both $C_P^{(2)} - \lam I$ and $C_P^{(1)} - \lam I$ are strong linearizations of $P(\lam)$, they have the same minimal polynomial $m(\lam)$. Let $\gamma = \deg m(\lam)$. By linearity, it suffices to check the statement for $v(x)=x^j$, $j=0,\dots,\gamma-1$.

We give an argument by induction. Note first that the base case, i.e., $v(x)=x^0=1$, is a trivial identity.
From the recurrence relation displayed in the proof of Barnett's theorem, we have that $\psi(\mathbb{DL}(P,1) (C_P^{(1)})^{j-1}) \equiv \mathcal{B}(P,x^{j-1} I)$. 
 By the inductive hypothesis we also have $\psi(\mathbb{DL}(P,1) (C_P^{(1)})^{j-1}) \equiv \psi((C_P^{(2)})^{j-1}\mathbb{DL}(P,1)) \equiv \mathcal{B}(P,y^{j-1} I)$. 

Now, let $\Delta(x,y)=\phi(\mathbb{DL}(P,1) (C_P^{(1)})^{j} - (C_P^{(2)})^j \mathbb{DL}(P,1))$. By the definitions of 
$\psi$ and $\phi$, for any block matrix $X$ we have $[\phi(X)] = \psi(X)$, where the notation $[\cdot]$ denotes an equivalence class $\modulo R_{P(y)}$ and $\modulo L_{P(x)}$ as in~\eqref{eq:eqclass}. 
More explicitly, for any bivariate matrix polynomial $S(x,y)$ in the equivalence class $\psi(X)$ there exist matrix polynomials $L(x,y)$, $R(x,y)$ and $C(x,y)$ such that 
\[\phi(X) = S(x,y) + L(x,y) P(x) + P(y) R(x,y) + P(y) C(x,y) P(x).\] 
Therefore, it must be 
\[\Delta(x,y) = (x-y)\mathcal{B}(P,x^{j-1} I) + L(x,y)P(x) + P(y) R(x,y) + P(y) C(x,y) P(x)\] for some $L(x,y)$, $R(x,y)$, and $C(x,y)$. But 
\[(x-y)\mathcal{B}(P,x^{j-1} I) = P(y) x^{j-1} - y^{j-1}P(x),\] and hence, $\Delta(x,y) \equiv 0 + L_1(x,y)P(x) + P(y)R_1(x,y) + P(y) C(x,y) P(x)$. 

Finally, observe that Theorem~\ref{biquotient} guarantees the existence and uniqueness of a matrix polynomial of grade $\deg P-1$ separately in $x$ and $y$ in the equivalence class $\psi(\phi^{-1}(\Delta(x,y)))$, and note that the latter must be equal to $\phi(\phi^{-1}(\Delta(x,y))=\Delta(x,y)$. On the other hand, $0$ has grade $\deg P -1$ separately in $x$ and $y$, and hence, $0=\Delta(x,y)$.
\end{proof}
\subsubsection{Properties of $\mathbb{BDL}(P,v)$}
We now investigate some properties of 
the  $\mathbb{BDL}(P,v)$ pencils defined in~\eqref{eq:BDLdef}. 
Clearly, an eigenvalue exclusion theorem continues to hold. Indeed, by assumption $\mathbb{DL}(P,1)$ is a linearization, because we suppose $P(x)$ has no eigenvalues at infinity. Thus, $\mathbb{BDL}(P,v)$ will be a linearization as long as $v(C_P^{(1)})$ is nonsingular, which happens precisely when $P(x)$ and $v(x) I$ do not share an eigenvalue. Nonetheless, it is less clear what properties, if any, pencils in $\mathbb{BDL}$ will inherit from pencils in $\mathbb{DL}$. Besides the theoretical interest of deriving its properties, $\mathbb{BDL}$ finds an application in the theory of the sign characteristics of structured matrix polynomials~\cite{signpaper}.
To investigate this matter, we 
will apply Theorem~\ref{quotient} taking $V(x)=v(x) I$. 

To analyze the implications of Theorem~\ref{quotient} and Theorem~\ref{biquotient}, it is worth summarizing the theory that we have built so far with a commuting diagram. Let $\mathbb{BDL}(P,v) = \lambda X + Y$ and $\mathbb{DL}(P,1) = \lambda \tilde{X} + \tilde{Y}$. Below, $F(x,y)$ 
(resp.~$\tilde{F}(x,y)$)
 denotes the continuous analogue of $X$ 
(resp.~$\tilde{X}$). 
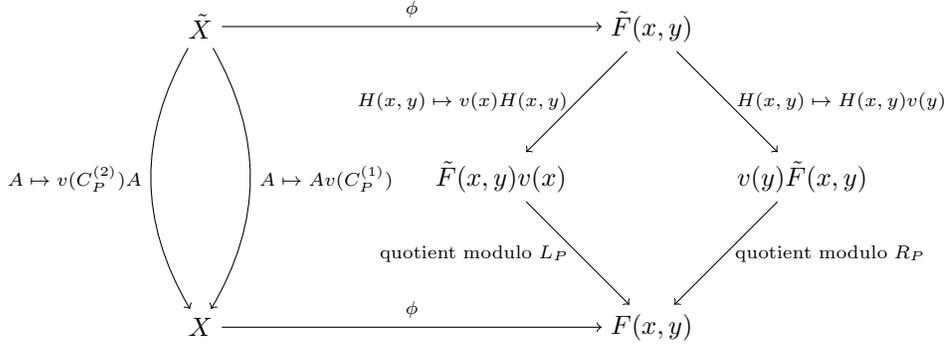
$$
\begin{tikzpicture}[scale=2]
\node (A) at (0,2) {$\tilde{X}$};
\node (B) at (3,2) {$\tilde{F}(x,y)$};
\node (C) at (2,1) {$\tilde{F}(x,y)v(x)$};
\node (D) at (0,0) {$X$};
\node (E) at (3,0) {$F(x,y)$};
\node (F) at (4,1) {$v(y)\tilde{F}(x,y)$};
\path[->,font=\scriptsize,]
(A) edge[bend left] node[right]{$A \mapsto A v(C_P^{(1)})$} (D)
(A) edge[bend right] node[left]{$A \mapsto v(C_P^{(2)}) A$} (D)
(C) edge node[left]{quotient modulo $L_P$} (E)
(A) edge node[above]{$\phi$} (B)
(D) edge node[above]{$\phi$} (E)
(B) edge node[left]{$H(x,y) \mapsto v(x)H(x,y)$} (C)
(B) edge node[right]{$H(x,y) \mapsto H(x,y)v(y)$} (F)
(F) edge node[right]{quotient modulo $R_P$} (E);
\end{tikzpicture} 
$$
An analogous diagram can be drawn for $Y$, $\tilde{Y}$, $G(x,y)$, and $\tilde{G}(x,y)$. The diagram above illustrates that we may work in the bivariate polynomial framework (right side of the diagram), which is often more convenient for algebraic manipulations than the matrix framework (left side). In particular, using Theorem~\ref{quotient}, Theorem~\ref{biquotient} and~\eqref{eq:DLrelations}, we obtain the following relations:
\begin{equation}\label{eq:shiftedsumbeyond}
v(y)P(x) \equiv S(y) P(x)\!=\!  F(x,y) x + G(x,y), \ \ yF(x,y) + G(x,y) \! =\! P(y)Q(x)\equiv P(y) v(x)
\end{equation}

In~\eqref{eq:shiftedsumbeyond}, $Q(x)$ and $S(y)$ are, respectively, the unique univariate matrix polynomials of grade $\deg P-1$ in $x$ (resp.~$y$) satisfying $v(x) I = Q(x) + A(x) P(x)$ (resp.~$v(y) I = S(y) + P(y) B(y)$ ) for some matrix polynomial $A(x)$ (resp.~$B(y)$). The existence and the uniqueness of $Q(x)$ and $S(y)$ follow from Theorem~\ref{quotient}.
To see how~\eqref{eq:shiftedsumbeyond} can be derived, take for example the second equation, as the argument is similar for the first one. By Lemma~\ref{lem:L1L2}, we have that $y \tilde{F}(x,y)v(x) + \tilde{G}(x,y) v(x) = P(y) v(x)$. Applying Theorem~\ref{quotient}, there is a unique matrix polynomial $Q(x)$ having grade $\deg P-1$ and such that $v(x) I = A(x) P(x) + Q(x)$. Hence, $y \tilde{F}(x,y)v(x) + \tilde{G}(x,y) v(x) = P(y) Q(x) + P(y) A(x) P(x)$. On the other hand, as illustrated by the diagram above $F(x,y) = \tilde{F}(x,y) v(x) + A(x,y)P(x)$ and similarly $G(x,y) = \tilde{G}(x,y) v(x) + B(x,y)P(x)$ for some bivariate matrix polynomials $A(x,y)$ and $B(x,y)$. Since $y F(x,y) + G(x,y)$ has grade $\deg P-1$ in $x$, by the uniqueness of the decomposition in Theorem~\ref{biquotient} we may conclude that $y F(x,y) + G(x,y) = P(y) Q(x)$.

From~\eqref{eq:shiftedsumbeyond} it appears clear that a pencil in $\mathbb{BDL}$ generally has \emph{distinct} left and right ansatz vectors, and that these ansatz vectors are now block vectors, associated with left and right ansatz matrix polynomials. For convenience of those readers who happen to be more familiar with the matrix viewpoint, we also display what we obtain by translating back~\eqref{eq:shiftedsumbeyond}:

\begin{equation}
X\css Y =  \begin{bmatrix}
S_{k-1}\\
\vdots\\
S_0
\end{bmatrix}
\left[P_k, P_{k-1},\ldots,P_0\right], \qquad
X\rss Y = 
\begin{bmatrix}
P_k\\
P_{k-1}\\
\vdots\\
P_0
\end{bmatrix} \left[Q_{k-1},\ldots,Q_0\right].
\end{equation}

Note that if $\deg v \leq k-1$ then $S(x)=Q(x)=v(x) I$ and we recover the familiar shifted sum equations for $\mathbb{DL}$.

The eigenvalue exclusion theorem continues to hold for $\mathbb{BDL}$ with a natural extension that replaces the ansatz vector $v$ with the matrix polynomial $Q$ (or $S$).
\begin{theorem}[Eigenvalue exclusion theorem for $\mathbb{BDL}$]
Let $P(x)$ be a matrix polynomial of degree $k$ with nonsingular leading coefficient, and let $v(x)$ be a scalar polynomial of arbitrary degree. Then, the pencil
$\mathbb{BDL}(P,v)$ defined as in~\eqref{eq:BDLdef} is a strong linearization of $P(x)$ if and only if $P(x)$ and $Q(x)$ (or $S(x)$) do not share an eigenpair, where $Q(x)$ and $S(x)$ are the unique matrix polynomials satisfying~\eqref{eq:shiftedsumbeyond}.
\end{theorem}

\begin{proof}
We prove the eigenvalue exclusion theorem for $P$ and $Q$, as the proof for $P$ and $S$ is analogous. We know that $\mathbb{BDL}(P,v)$ is a strong linearization if and only if we cannot find an eigenvalue $x_0$ and a nonzero vector $w$ such that $P(x_0)w=v(x_0)I w=0$. (Here, we are implicitly using the fact that if $x_0$ is an eigenvalue of $v(x) I$, then any nonzero vector is a corresponding eigenvector.)  But in the notation of Theorem~\ref{quotient}, we can write uniquely $Q(x)=v(x)I-A(x)P(x)$, and hence, $Q(x_0)w=v(x_0)w-A(x_0)P(x_0)w$. Hence, $P(x)$ and $v(x)I$ share an eigenpair if and only if $P(x)$ and $Q(x)$ do.
\end{proof}

We now show that pencils in $\mathbb{BDL}$ still are Lerer--Tismenetsky B\'{e}zoutians. It is convenient to first state a lemma and a corollary.

\begin{lemma}\label{Toeplitz2}
Let $U \in \mathbb{F}^{nk \times nk}$ be an invertible block-Toeplitz upper-triangular matrix. Then $(U^{\mathcal{B}})^{-1} = (U^{-1})^{\mathcal{B}}$.
\end{lemma}
\begin{proof}
We claim that, more generally, if $U$ is an invertible Toeplitz upper-triangular matrix with elements in any ring with unity, and $L=U^T$, then $U^{-1}=(L^{-1})^T$. Taking $\mathbb{F}^{n \times n}$ as the base ring yields the statement. To prove the claim, recall that if $L^{-1}$ exists then $L^{-1}=L^{\#}$, where the latter notation denotes the group inverse of $L$. Explicit formulae for $L^{\#}$ appeared in~\cite[eq.~3.4]{Hartwig_77}\footnote{It should be noted that if $L^{-1}$ exists then $L_{11}$ must be invertible too, where $L_{11}$ denotes the top-left element of $L$: if the base ring is taken to be $\mathbb{F}^{n \times n}$, that is the $n \times n$ top-left block of $L$.  Moreover,~\cite[Theorem 2]{Hartwig_77} implies that ~\cite[eq.~3.2]{Hartwig_77} is satisfied.}. Hence, it can be checked by direct computation that $(L^{-1})^T U = U (L^{-1})^T= I$.
\end{proof}

\begin{corollary}\label{Toeplitz}
Let $U \in \mathbb{F}^{nk \times nk}$ be an invertible block-Toeplitz upper-triangular matrix and $\Upsilon=\begin{bmatrix} v_1 I_n\\
\vdots\\
v_k I_n\end{bmatrix}$, $v_i \in \mathbb{F}$.
 Then $(U^{-1} \Upsilon)^{\mathcal{B}} = \Upsilon^{\mathcal{B}} 
(U^{\mathcal{B}})^{-1}$.
\end{corollary}
\begin{proof}
Since the block elements of $\Upsilon$ commute with any other matrix, it suffices to apply Lemma~\ref{Toeplitz2}.
\end{proof}

\begin{theorem}\label{thm:PQ=SP}
If $Q(x)$, $A(x)$, $S(x)$, and $B(x)$ are defined as in Theorem~\ref{quotient} with $V(x)=v(x) I$, then $P(x) Q(x) = S(x) P(x)$ and $A(x)=B(x)$.
\end{theorem}

\begin{proof}
Let $v(x)I-Q(x)=A(x)P(x)$ and $v(x)I-S(x)=P(x)B(x)$. We may assume $\deg v \geq k$, as otherwise
the statement is trivially verified since $Q(x)=S(x)=v(x)I$ and $A(x)=B(x)=0$. Note first that $\deg A = \deg B =\deg v -k$ because by assumption the leading coefficient of $P(x)$ is not a zero divisor. The  coefficients of $A(x)$ must satisfy~\eqref{eq:whatisA}, while block transposing~\eqref{eq:whatisA} we obtain an equation that must be satisfied by the coefficients of $B(x)$. Equating term by term and using Corollary~\ref{Toeplitz} we obtain $A(x)=B(x)$, and hence, $P(x)Q(x)-S(x)P(x)=P(x)B(x)P(x)-P(x)A(x)P(x)=0$.
\end{proof}

Hence, it follows that 
$\mathbb{BDL}(P,v)$ is a Lerer--Tismenetsky B\'{e}zoutian (compare the result with~\eqref{eq:Lbezfunc}). We now present the following immediate corollary: 
\begin{corollary}\label{cor:whatisBDL}
It holds 
\[\mathbb{BDL}(P,v) =\lambda B_{S,P}(Q,P) + B_{P,xS}(P,xQ),\]
where $Q(x)$ and $S(x)$ are as in Theorem~\ref{quotient} the unique matrix polynomials of grade $k-1$ satisfying $v(x) I = P(x) A(x) + S(x) = A(x) P(x) + Q(x)$ for some matrix polynomial $A(x)$.
\end{corollary}
\begin{proof}
Observe first that by Definition~\ref{def:bezmatpoly} $B_{S,P}(Q,P)$ and $B_{P,xS}(P,xQ)$ are well defined since Theorem~\ref{thm:PQ=SP} implies that $P(x)Q(x)=S(x)P(x)$ and $xS(x)P(x)=P(x)xQ(x)$. The proof of the corollary is then a straightforward application of~\eqref{eq:shiftedsumbeyond}. For example, for the leading term we have from~\eqref{eq:shiftedsumbeyond} that $F(x,y) \!=\! \frac{S(y)P(x) - P(y)Q(x)}{x-y} \!=\! \mathcal{B}_{S,P}(Q,P)$. Translating back from bivariate polynomial to block matrices, we find that $\phi^{-1}(F(x,y))=B_{S,P}(Q,P)$. The proof for the trailing term of the pencil is analogous and we omit the details.
\end{proof}

Once again, if $\deg v \leq k-1$ then we recover $\mathbb{DL}(P,v)$ because $S(x)=Q(x)=v(x)I$. More generally, we have $S(x)-Q(x)=[A(x),P(x)]:=A(x)P(x)-P(x)A(x)$.

For the rest of this section, we assume that the underlying field $\mathbb{F}$ is a metric space; for simplicity, we focus on the case $\mathbb{F}=\mathbb{C}$. 
As mentioned in Section~\ref{sec:genpolybases}, one property of a pencil in $\mathbb{DL}$ is block symmetry. It turns out that this property does not hold for pencils in $\mathbb{BDL}$. Nonetheless, an even deeper algebraic property is preserved. Since each matrix coefficient in a pencil in $\mathbb{DL}$ is a 
 B\'{e}zout matrix, the inverses of those matrices are block Hankel -- note that unless $n=1$, the inverse of a block Hankel matrix needs not be block symmetric. The general result is: a matrix is the inverse of a block Hankel if and only if it is a 
 B\'{e}zout matrix~\cite[Corollary 3.4]{LT2}. However, for completeness, we give a simple proof for the special case of our interest.

\begin{theorem}
Let $\lambda X + Y$ be a pencil either in $\mathbb{DL}$ or in $\mathbb{BDL}$ associated with a matrix polynomial $P(x) \in \mathbb{C}[x]^{n \times n}$ with an invertible leading coefficient. 
Then, $X^{-1}$ and $Y^{-1}$ are both block Hankel matrices if the inverses exist. 
\end{theorem}
\begin{proof}
Note first that, by Lemma~\ref{lem:Lbzefunc}, it suffices to show that $B(P,v I)^{-1}$ is block Hankel for all polynomials $v$ such that the inverse exists. 

Assume first $P(0)$ is invertible, implying that $C_P^{(1)}$ is invertible as well.
We have that $H_0=(B(P,I))^{-1}$ is block Hankel, as can be easily shown by induction on $k$~\cite[Sec.~2.1]{Gohberg_09_01}. By Barnett's theorem, $(C_P^{(2)})^j B(P,I) = B(P,I) (C_P^{(1)})^j$.  Then 
$H_j:=(C_P^{(1)})^{-j} H_0 = H_0 (C_P^{(2)})^{-j}$. Taking into account the structure of $(C_P^{(1)})^{-1}$ and $(C_P^{(2)})^{-1}$, we see by induction that $H_j$ is block Hankel. For a general $v(x)$ such that $v(x) I$ does not share eigenvalues with $P(x)$, we have that $(B(P,v I))^{-1}=v(C_P^{(1)})^{-1} H_0$. Since $v(C_P^{(1)})^{-1}$ is a polynomial in $(C_P^{(1)})^{-1}$, this is a linear combination of the $H_j$, hence is block Hankel. 

If $P(0)$ is singular consider any sequence
 $(P_n)_{n \in \mathbb{N}} = P(x) + E_n$ such that $\|E_n\| \rightarrow 0$ 
as $n\rightarrow\infty$
and $P_n(0)=P(0)+E_n$ is invertible for all $n$ (such a sequence exists because singular matrices are nowhere dense). Since the 
 B\'{e}zout matrix is linear in its arguments, $B(P_n,v I) \rightarrow B(P,v I)$. In particular, $B(P_n,v I)$ is eventually invertible if and only if no root of $v(x)$ is an eigenvalue of $P(x)$. The inverse is continuous as a matrix function, and thus $B(P,v I)^{-1} = \lim_{n \rightarrow \infty} B(P_n,v I)^{-1}$. We conclude by observing that the limit of a sequence of block Hankel matrices is block Hankel.
\end{proof}

Note that the theorem above implies that if $\lambda_0$ is not an eigenvalue of $P$ then the evaluation of a linearization in $\mathbb{DL}$ or $\mathbb{BDL}$ at $\lambda=\lambda_0$ is the inverse of a block Hankel matrix.

\subsubsection{$\mathbb{BDL}(P,v)$ and structured matrix polynomials}
We now turn to exploring the connections between $\mathbb{BDL}(P,v)$ and structured matrix polynomials. 
Recall that a Hermitian matrix polynomial is a polynomial whose coefficients are all Hermitian matrices. If $P(x)$ is Hermitian we write $P^*(x)=P(x)$.
It is often argued that block-symmetry is important because, if $P(x)$ was Hermitian in the first place and $v(x)$ has real coefficients, then $\mathbb{DL}(P,v)$ is also Hermitian. Although $\mathbb{BDL}(P,v)$ is not block-symmetric, it still is Hermitian when $P(x)$ is Hermitian. 

\begin{theorem}\label{hermitian}
Let $P(x) \in \mathbb{C}^{n \times n}[x]$ be a Hermitian matrix polynomial with invertible leading coefficient and $v(x) \in \mathbb{R}[x]$ a scalar polynomial with real coefficients. Then, $\mathbb{BDL}(P,v)$ is a Hermitian pencil.
\end{theorem}

\begin{proof}
Recalling the explicit form of $\mathbb{BDL}(P,v) = \lambda X + Y$ from Corollary~\ref{cor:whatisBDL}, we have
$X= B_{S,P}(Q,P)$ and $Y=B_{P,xS}(P,xQ)$. Here, $Q(x)$ (resp.~$S(x)$) are as in Theorem~\ref{quotient} the unique matrix polynomials of grade $k-1$ such that $v(x) I = A(x) P(x) + S(x) = P(x) A(x) Q(x)$, where $A(x)$ is also unique and we are using Theorem~\ref{thm:PQ=SP} as well. Then $-X$ is associated with the 
Lerer--Tismenetsky
B\'{e}zoutian function $F(x,y)=\frac{P(y)Q(x) - S(y) P(x)}{x-y}$. By definition, $S(x) = v(x) I - P(x) A(x)$. Taking the transpose conjugate of this equation, and noting that by assumption $P(x)=P^*(x)$, $v(x)=v^*(x)$, we obtain $S^*(x)=v(x) I - A^*(x) P(x)$. But, by Theorem~\ref{quotient}, there is a \emph{unique} matrix polynomial $Q(x)$ of grade $k-1$ such that $v(x) I = Q(x) + A(x) P(x)$ for some $A(x)$. Thus, since $\deg S^*(x) = \deg S(x) \leq k-1$, we conclude that  $S^*(x)=Q(x)$. (Although not strictly needed in this proof, the uniqueness of $A(x)$ also implies $A^*(x) = A(x)$.) Hence, $F(x,y)=\frac{P(y)Q(x)-Q^*(y)P(x)}{x-y} = \frac{Q^*(y)P(x) - P(y) Q(x)}{y-x} = F^*(y,x)$, proving that $X$ is Hermitian because the formula holds for any $x,y$. 

Analogously $G(x,y)=\frac{P(y)xQ(x)-yQ^*(y)P(x)}{x-y}=\frac{yQ^*(y)P(x)-P(y)xQ(x)}{y-x}=G^*(y,x)$, allowing us to deduce that $Y$ is also Hermitian.
\end{proof}

The theory of functions of a matrix~\cite{high:FM} allows one to extend the definition of $\mathbb{BDL}$ to a general function $f$, rather than just a polynomial $v$, as long as $f$ is defined on the spectrum of $C_P^{(1)}$ (for a more formal definition see~\cite{high:FM}). One just puts $\mathbb{BDL}(P,f):=\mathbb{BDL}(P,v)$ where $v(x)$ is the interpolating polynomial such that $v(C_P^{(1)})=f(C_P^{(1)})$.

\begin{corollary}
Let $P(x) \in \mathbb{C}^{n \times n}[x]$ be 
a Hermitian matrix polynomial with invertible leading coefficient
 and $f: \mathbb{C} \rightarrow \mathbb{C}$ a function defined on the spectrum of $C_P^{(1)}$ and such that 
$f(x^*)=(f(x))^*$.
Then $\mathbb{BDL}(P,f)$ is a Hermitian pencil.
\end{corollary}
\begin{proof}
It suffices to observe that the properties of $f$ and $P$ imply that $f(C_P^{(1)})=v(C_P^{(1)})$ with $v\in \mathbb{R}[x]$~\cite[Def.~1.4]{high:FM}. 
\end{proof}

In the monomial basis, other structures of interest have been defined, such as $*$-even, $*$-odd, $T$-even, $T$-odd (all these definitions can be extended to any alternating basis, such as Chebyshev) or $*$-palindromic, $*$-antipalindromic, $T$-palindromic, $T$-antipalindromic. For $\mathbb{DL}$, analogues of Theorem~\ref{hermitian} can be stated in all these cases~\cite{goodvibrations}. These properties extend to $\mathbb{BDL}$. We state and prove them for the $*$-even and the $*$-palindromic case: 

\begin{theorem}\label{structures}
Assume that 
$P(x)=P^*(-x)$ is $*$-even 
and with an invertible leading coefficient,
and that $f(x)=f^*(-x)$, and let 
$\Sigma$ be as in~\eqref{eq:sigR}. 
Then $\Sigma \mathbb{BDL}(P,f)$ is a $*$-even pencil.
Furthermore, if $P(x)=x^{k} P^*(x^{-1})$ is $*$-palindromic and $f(x)=x^{k-1}f(x^{-1})$, and 
defining the ``flip matrix'' $R$ as in \eqref{eq:sigR}, then $R \mathbb{BDL}(P,f)$ is a $*$-palindromic pencil.
\end{theorem}


\begin{proof}
The proof goes along the same lines as that of Theorem~\ref{hermitian}: we first use the functional viewpoint and the B\'{e}zoutian interpretation of $\mathbb{BDL}(P,v)=\lam X + Y$ (see in particular Corollary~\ref{cor:whatisBDL}) to map $\phi(X)=F(x,y) = \frac{-P(y)Q(x)+S(y)P(x)}{x-y}$ and 
$\phi(Y)=G(x,y)=\frac{P(y)xQ(x)-yS(y)P(x)}{x-y}$. Once again, here $Q(x)$ and $S(x)$ are as in Theorem~\ref{quotient} the unique matrix polynomials of grade $k-1$ such that $v(x) I = A(x) P(x) + S(x) = P(x) A(x) Q(x)$, where $A(x)$ is also unique and we are using Theorem~\ref{thm:PQ=SP} as well. Assume first that $P$ $*$-even: we claim that $\lam  X^* \Sigma +  Y^* \Sigma = -\lam \Sigma X + \Sigma Y$.  Indeed, note that $v(x)$, the interpolating polynomial of $f(x)$, must also satisfy $v^*(x)=v(-x)$. Taking the transpose conjugate of the equation $S(x)=v(x) I - P(x) B(x)$, and using Theorem~\ref{quotient} as in the proof of Theorem~\ref{hermitian}, we obtain $Q^*(x)=S(-x)$. This, together with Table~\ref{tab:op2}, implies that $\phi(-\Sigma X) = \frac{P(-y)Q(x)-Q^*(y)P(x)}{x+y} = -\frac{Q^*(y)P^*(-x)-P^*(y)Q(x)}{x+y} = \phi( X^* \Sigma)$.
Similarly, $\phi(\Sigma Y) = \frac{P(-y)xQ(x)+yQ^*(y)P(x)}{x+y} = \frac{yQ^*(y)P^*(-x)+P^*(y)xQ(x)}{y+x} = \phi( Y^* \Sigma)$.

The case of a $*$-palindromic $P$ is dealt with analogously and we omit the details.
\end{proof}

Similar statements hold for other structures. We summarize them in the following table, omitting the proofs as they are completely analogous to those of theorems~\ref{hermitian} and~\ref{structures}.

\begin{table}[htbp]
  \centering
  \caption{Structures of $P$, $\deg P=k$, and potential linearizations that are structure-preserving}
  \label{tab:stru}
\begin{tabular}{c|c|c}
Structure of $P$ & Requirement on $f$ & Pencil\\
\hline
Hermitian: $P(x)=P^*(x)$ & $f(x^*)=f^*(x)$ & $\mathbb{BDL}(P,f)$\\
skew-Hermitian: $P(x)=-P^*(x)$ & $f(x^*)=f^*(x)$  & $\mathbb{BDL}(P,f)$ \\
symmetric: $P(x)=P^T(x)$ & any $f(x)$ & $\mathbb{BDL}(P,f)$ \\
skew-symmetric: $P(x)=-P(x)^T$ & any $f(x)$ & $\mathbb{BDL}(P,f)$ \\
\hline
*-even: $P(x)=P^*(-x)$ & $f(x)=f^*(-x)$ & $\Sigma \mathbb{BDL}(P,f)$ \\
*-odd: $P(x)=-P^*(-x)$ & $f(x)=f^*(-x)$ & $\Sigma \mathbb{BDL}(P,f)$ \\
T-even: $P(x)=P^T(-x)$ & $f(x)=f(-x)$ & $\Sigma \mathbb{BDL}(P,f)$ \\
T-odd: $P(x)=-P^T(-x)$ & $f(x)=f(-x)$ & $\Sigma \mathbb{BDL}(P,f)$ \\
\hline
*-palindromic: $P(x)=x^kP^*(x^{-1})$ & $f(x)=x^{k-1} f^*(x^{-1})$& $R \mathbb{BDL}(P,f)$\\
*-antipalindromic: $P(x)=-x^{k}P^*(x^{-1})$ & $f(x)=x^{k-1} f^*(x^{-1})$& $R \mathbb{BDL}(P,f)$\\
T-palindromic: $P(x)=P^*(x^{-1})$ & $f(x)=x^{k-1} f(x^{-1})$& $R \mathbb{BDL}(P,f)$\\
T-antipalindromic: $P(x)=-P^T(x^{-1})$ & $f(x)=x^{k-1} f(x^{-1})$& $R \mathbb{BDL}(P,f)$\\
\end{tabular}
\end{table}

With a similar technique, one may produce pencils with a structure that is \emph{related} to that of the linearized matrix polynomial, e.g., if $P$ is $*$-odd and $f(x)=-f^*(-x)$, then $\Sigma \mathbb{BDL}(P,f)$ will be $*$-even. For lack of space we will not include a complete list of such variations on the theme in this paper. However, we note that generalizations of this kind are immediate to prove with the Lerer--Tismenetsky B\'{e}zoutian functional approach.

We conclude this section by giving the following result which has an application in the theory of sign characteristics~\cite{signpaper}: 

\begin{theorem}
Let $P(x)$ be $*$-palindromic of degree $k$, with nonsingular leading coefficient, and $f(x)=x^{k/2}$; if $k$ is odd, suppose furthermore that the square root is defined in such a way that $P(x)$ has no eigenvalues on the branch cut. Moreover, let $\mathbb{BDL}(P,f)=\lambda X + Y$ and let $R$ be defined as in~\eqref{eq:sigR}. Then $Z=i R X$ is a Hermitian matrix.
\end{theorem}
\begin{proof}
We claim that the statement is true when $P(x)$ has all distinct eigenvalues. Then it must be true in general. This follows by continuity, if we consider a sequence $(P_n(x))_n$ of $*$-palindromic polynomials converging to $P(x)$ and such that $P_n(x)$ has all distinct eigenvalues, none of which lie on the branch cut. Such a sequence exists because the set of palindromic matrix polynomials with distinct eigenvalues is dense, as can be seen arguing on the characteristic polynomial seen as a polynomial function of the $n^2(k+1)$ independent real parameters.

It remains to prove the claim.
Since $X$ is the linear part of the pencil $\mathbb{BDL}(P,f)$, we get, by Corollary~\ref{cor:whatisBDL} and using the mapping $\phi$ defined in Section~\ref{sec:bivariate}, that $\phi(X)= \frac{P(y)Q(x)-S(y)P(x)}{x-y}$, where $v(x) I = Q(x) + A(x) P(x) = S(x) + P(x) A(x)$ are defined as in Theorem~\ref{quotient} and $v(x)$ is the interpolating polynomial of $f(x)$ on the eigenvalues of $P(x)$. By assumption $P(x)$ has $kn$ distinct eigenvalues. 
Denote by $(\lambda_i,w_i,u_i)$, $i=1,\ldots,nk$, an eigentriple, and consider the matrix in Vandermonde form $V$ whose $i$th column is
 $V_i=\Lambda(\lambda_i) \otimes w_i$ ($V$ is the matrix of eigenvectors of $C_P^{(1)}$); 
recall moreover that if $(\lambda_i,w_i,u_i)$ is an eigentriple then $(1/\lambda^*_i,u_i^*,w_i^*)$ is. Observe that by definition $Q(\lambda_i) w_i = \lambda_i^{k/2} w_i$ and $u_i S(\lambda_i) = u_i \lambda_i^{k/2}$. 

Our task is to prove that $R X = - X^* R$; observe that this is equivalent to $V^* R X V = - V^* X^* R V$. Using Table~\ref{tab:op} and Table~\ref{tab:op2}, we see that $V_i^* R X V_j$ is equal to the evaluation of $w_i^* \frac{y^k P(1/y)Q(x)-y^k S(1/y)P(x)}{x y-1} w_j$  at $(x=\lambda_j,y=\lambda_i^*)$. Suppose first that $\lambda_i \lambda_j^* \neq 1$. Then, using $P(\lambda_j) w_j = 0$ and $w_i^* P(1/\lambda_i^*)=0$, we get $V_i^* R X V_j=0$. When $\lambda_i^{-1} = \lambda_j^*$, we can evaluate the fraction using De L'H\^{o}pital rule, and obtain $w_i^* \frac{-(\lambda_i^*)^k S(1/\lambda_i^*)P'(\lambda_j)}{\lambda_i^*} w_j = - w_i^* (\lambda_i^*)^{k/2-1} P'(\lambda_j) w_j$. An argument similar to the previous one shows that 
$V_i^* X^* R V_j=0$ when $\lambda_i \lambda_j^* \neq 1$, and
 $V_i^* X^* R V_j=w_i^* (\lambda_i^*)^{k/2-1} P'(\lambda_j) w_j$ 
when $\lambda_i \lambda_j^* = 1$.

We have thus shown that 
$V_i^* X^* R V_j = -V_i^*  R X V_j$ for all $(i,j)$, establishing the claim. 
\end{proof}

\section{Conditioning of eigenvalues of $\mathbb{DL}(P)$}\label{sec:conditioning}
In~\cite{Higham_06_02}, a conditioning analysis is carried out for the
eigenvalues of the $\mathbb{DL}(P)$ pencils, which identifies
situations in which the $\mathbb{DL}(P)$ linearization itself
does not worsen the eigenvalue conditioning of the original matrix
polynomial $P(\lambda)$ expressed in the monomial basis.

Here, we use the bivariate polynomial viewpoint to analyze the
conditioning, using concise arguments and allowing for $P(\lambda)$
expressed in any polynomial basis.  As shown in~\cite{tisseur2000backward}, the first-order expansion of a simple
eigenvalue $\lambda_i$ of $P(\lambda)+\Delta P(\lambda)$ is
\begin{equation} \label{eq:matpolyeigpert} \lambda_i=\lambda_i(P) -
  \frac{y_i^*\Delta
    P(\lambda_i)x_i}{y_i^*P'(\lambda_i)x_i}+\mathcal{O}(\|\Delta
  P(\lambda_i)\|^2),
\end{equation}
where $y_i$ and $x_i$ are the left and right eigenvectors
corresponding to $\lambda_i$. This analysis motivated the conditioning 
results for multidimensional rootfinding~\cite{Nakatsukasa_13_01,noferini2015resultant}. 

When applied to a $\mathbb{DL}(P)$ pencil $L(\lambda)=\lambda X+Y$
with ansatz $v$, defining $\widehat x_i = \Lambda(\lam_i)\otimes x_i,
\widehat y_i =  \Lambda^T(\lam_i) \otimes y_i$, where $\Lambda(\lam) =
\left[\phi_{k-1}(\lam),\ldots,\phi_0(\lam)\right]^T$ as before and
noting that $L'(\lambda) = X$,~\eqref{eq:matpolyeigpert} becomes
\begin{equation}
  \label{eq:dlpert}
  \lambda_i=\lambda_i(L) - \frac{\widehat y_i^*\Delta L(\lambda_i)\widehat x_i}{\widehat y_i^*X\widehat x_i}+\mathcal{O}(\|\Delta L(\lambda_i)\|^2),
  \quad 
  i=1,\ldots, nk. 
\end{equation}
Recall from~\eqref{eq:Lbezfunc} that 
$X =-B(P,v) =  B(v,P)$, 
and note in Table~\ref{tab:op} that $\widehat
y_i^*X\widehat x_i$ is the evaluation of the $n\times n$ 
Lerer--Tismenetsky
\bezoutian\
function $\mathcal{B}(v,P))$ setting both variables equal to $\lambda_i$, followed
by left and right multiplication by $y_i^*$ and $x_i$. Therefore, 
since the Lerer--Tismenetsky \bezoutian\ function is a polynomial, hence continuous with respect to its arguments, 
we have
\begin{align*}
  \widehat y_i^*X\widehat x_i&= y_i^*\left(\lim_{s,t\rightarrow
      \lambda_i}\frac{v(s)P(t)-P(s)v(t)}{s-t} \right)x_i
  \\
  &= y_i^*\left(v'(\lambda_i)P(\lambda_i)-P'(\lambda_i)v(\lambda_i)\right)x_i\\
  &= -y_i^*P'(\lambda_i)v(\lambda_i)x_i.
\end{align*}
Here we used L'H\^{o}pital's rule for the second equality and
$P(\lambda_i)x_i=0$ for the last.

Hence, the expansion~\eqref{eq:dlpert} becomes
\begin{equation} \label{eq:dlpert2} \lambda_i=\lambda_i(L) +
  \frac{1}{v(\lambda_i)}\frac{\widehat y_i^*\Delta
    L(\lambda_i)\widehat
    x_i}{y_i^*P'(\lambda_i)x_i}+\mathcal{O}(\|\Delta
  L(\lambda_i)\|^2).
\end{equation}
Thus, up to first order, a small change 
of $L$ to $L+\Delta L$ perturbs $\lambda_i$ by 
$\frac{|\widehat y_i^*\Delta    L(\lambda_i)\widehat    x_i|}{|v(\lambda_i)||y_i^*P'(\lambda_i)x_i|}\leq \frac{\|\widehat y_i\|_2\|\Delta    L(\lambda_i)\|_2\|\widehat    x_i\|_2}{|v(\lambda_i)||y_i^*P'(\lambda_i)x_i|}$, where the last inequality is sharp in that equality can hold by taking $\Delta L(\lambda)=\sigma\widehat y_i\widehat x_i^*$ for any scalar $\sigma$. 

%

Similarly from~\eqref{eq:matpolyeigpert}, 
a small perturbation from $P$ to $P+\Delta P$ results in the eigenvalue perturbation $\frac{\|y_i\|_2\|\Delta    P(\lambda_i)\|_2\|x_i\|_2}{|y_i^*P'(\lambda_i)x_i|}$, which is also a sharp bound. 
Combining these two bounds, we see that the ratio between the perturbation of $\lambda_i$ in the original $P(\lambda)$ and the linearization $L(\lambda)$ is
\begin{equation}
  \label{eq:rratio}
  r_{\lambda_i}=\frac{1}{v(\lambda_i)}\frac{\|\widehat y_i\|_2\|\Delta L(\lambda_i)\|_2\|\widehat x_i\|_2}{\|y_i\|_2\|\Delta P(\lambda_i)\|_2\|x_i\|_2}.   
\end{equation}

Now recall that the absolute \emph{condition number} of an eigenvalue of a matrix polynomial may be defined as
\begin{equation}  \label{eq:conddef}
\kappa(\lambda) = \lim_{\epsilon \rightarrow 0} \sup \{|\Delta \lambda| : \left( P(\lambda + \Delta \lambda) + \Delta P (\lambda + \Delta \lambda) \right) \!\hat x = 0, \hat x \neq 0, \| \Delta P(\cdot) \| \leq \epsilon \| P(\cdot)\| \}.
\end{equation}
Here, we are taking the norm for matrix polynomials to be $\|P(\cdot)\|=\max_{\lambda\in\mathcal{D}}\|P(\lambda)\|_2$, where $\mathcal{D}$ is the domain  of interest that below we take to be the interval $[-1,1]$. 
In~\eqref{eq:conddef}, $\lambda + \Delta \lambda$ is the eigenvalue of $P+\Delta P$ closest to $\lambda$ such that $\lim_{\epsilon\rightarrow 0}\Delta\lambda=0$. 
Note that definition~\eqref{eq:conddef} is the \emph{absolute} condition number, in contrast to the relative condition number treated in~\cite{tisseur2000backward}, in which the supremum is taken of $|\Delta \lambda|/(\epsilon|\lambda|)$, and over $\Delta P(\cdot) = \sum_{i=0}^k\Delta P_i\phi_i(\cdot)$  such that $\|\Delta P_i\|_2\leq \epsilon \| E_i\|$ where $E_i$ are prescribed tolerances for the term with $\phi_i$. 
Combining this definition with the analysis above, we can see that the ratio of the condition numbers of the eigenvalue $\lambda$ for the linearization $L$ and the original matrix polynomial $P$ is
\begin{equation}  \label{eq:rrratio}
  \widehat r_{\lambda_i}=\frac{1}{v(\lambda_i)}\frac{\|\widehat y_i\|_2\|L(\cdot)\|\|\widehat x_i\|_2}{\|y_i\|_2\|P(\cdot)\|\|x_i\|_2}.   
\end{equation}
The eigenvalue $\lambda_i$ can be computed stably from the
linearization $L(\lambda)$ if $\hat r_{\lambda_i}$ is not significantly
larger than 1.  Identifying conditions to guarantee
$\hat r_{\lambda_i}=\mathcal{O}(1)$ is nontrivial and depends not only on
$P(\lambda)$ and the choice of the ansatz $v$, but also on the value
of $\lambda_i$ and the choice of polynomial basis.  For example,~\cite{Higham_06_02} 
considers the monomial case and shows that the coefficientwise
conditioning of $\lambda_i$ does not worsen much by forming
$L(\lambda)$ if $\frac{\max_i\|P_i\|_2}{\max\{\|P_0\|_2,\|P_k\|_2\}}$
is not too large, where $P(\lam) = \sum_{i=0}^k P_i\lambda^i$, and the
ansatz choice is $v=\lambda^{k-1}$ if $|\lambda_i|\geq 1$ and $v=1$ if
$|\lambda_i|\leq 1$.

Although it is difficult to make a general statement on when
$r_{\lambda_i}$ is moderate, 
here we show that in the
practically important case where the Chebyshev basis is used and
$\lambda_i\in \mathcal{D}:=[-1,1]$, 
the conditioning ratio can be bounded by 
a modest polynomial in $n$ and $k$, 
with an
appropriate choice of $v$, namely,  $v=1$. This means that the conditioning of these eigenvalues 
does not worsen much by forming the
linearization, and the eigenvalues can be computed in a stable manner from $L(\lambda)$.

\begin{theorem}
Let  $L(\lambda)$ be the $\mathbb{DL}(P)$ linearization  with ansatz $v(x)=1$
of a matrix polynomial $P(\lambda)$ expressed in the Chebyshev basis $\phi_j(x)=T_j(x)$. 
Let $\lambda_i$ be an eigenvalue of $P(\lambda)$ with right and left eigenvectors $x_i$ and $y_i$, respectively, such that $P(\lambda_i)x_i=0$, $y_i^TP(\lambda_i)=0$, and define 
$\widehat x_i = x_i\otimes \Lambda(\lam_i), \widehat y_i = y_i\otimes \Lambda(\lam_i)$ where $\Lambda(\lam) =
\left[T_{k-1}(\lam),\ldots,T_0(\lam)\right]^T$. 
Then 
for any eigenvalue $\lambda_i\in[-1,1]$, 
the conditioning ratio $\widehat r_{\lambda_i}$ in~\eqref{eq:rrratio} 
 is bounded by
\begin{equation}
  \label{eq:condrate}
 \widehat r_{\lambda_i}\leq 16n(e-1)k^4. 
\end{equation}
\end{theorem}
\begin{proof}
Since the Chebyshev polynomials 
$T_j(x)$ are all bounded by $1$ on $[-1,1]$, we
have $\|\hat x_i\|_2=c_i\| x_i\|_2$, $\|\widehat y_i\|_2=d_i\|
y_i\|_2$ for some $c_i,d_i\in[1,\sqrt{k}]$.  Therefore, we have
\begin{align}
 \widehat  r_{\lambda_i}
  &\leq \frac{k}{v(\lambda_i)}\frac{\|L(\cdot)\|}{\| P(\cdot)\|}. \label{eq:rlam}
\end{align}

We next claim that $\| L(\cdot)\|$ can be estimated as $\|
L(\cdot)\|=\mathcal{O}(\|P(\cdot)\|\|v(\cdot)\|)$.
To verify this it suffices to show that writing $L(\lambda) =
  \lambda X+Y$ we have 
  \begin{equation}
    \label{eq:xysmall}
    \|X\|_2\leq q_X(n,k)\|P(\cdot)\|\|v(\cdot)\|  ,\quad 
    \|Y\|_2\leq q_Y(n,k)\|P(\cdot)\|\|v(\cdot)\|  
  \end{equation}
  where $q_X,q_Y$ are low-degree polynomials with modest coefficients.
  Let us first prove the bound for $\|X\|_2$ in~\eqref{eq:xysmall} (to
  gain a qualitative understanding one can consult the construction of
  $X,Y$ in Section~\ref{sec:construction}).  

 Recalling~\eqref{eq:Lbezfunc}, $X$ is the B\'{e}zout block matrix
  $B(vI,P)$, so its $(k-i,k-j)$ block is the coefficients 
for $T_{i}(y)T_j(x)$
of the function
  \[ \mathcal{B}(P,-vI)=\frac{-P(y)v(x) +v(y)P(x)}{x-y} :=H(x,y). \]
  Recall that $H(x,y)$ is an $n\times n$ bivariate matrix polynomial, and denote
  its $(s,t)$ element by $H_{st}(x,y)$.  For every fixed value of
  $y\in[-1,1]$, by~\cite[Lem.~B.1]{NNvandooren14} 
  we have
  \[
  |H_{st}(x,y)|\leq(e-1)k^2\max_{x\in[-1,1]}|H_{st}(x,y)(x-y)| \leq
  2(e-1)k^2 \|P(\cdot)\|\|v(\cdot)\|  
  \]
  when $|x-y|\leq k^{-2}$. Clearly,
  \[
  |H_{st}(x,y)|\leq 2k^2 \|P(\cdot)\|\|v(\cdot)\|\quad \mbox{for}
  \quad |x-y|\geq k^{-2}.
  \]
  Together we obtain $\max_{x\in[-1,1]}|H_{st}(x,y)|\leq 2(e-1)k^2
  \|P(\cdot)\|\|v(\cdot)\|$. Since this holds for every $(i,j)$ and
  every fixed value of $y\in[-1,1]$ we obtain
  \begin{equation} \label{eq:Hbound}
    \max_{x\in[-1,1],y\in[-1,1]}|H_{st}(x,y)|\leq 2(e-1)k^2
    \|P(\cdot)\|\|v(\cdot)\|.
  \end{equation}
  To obtain~\eqref{eq:xysmall} it remains to bound the 
coefficients in the
  representation of a degree-$k$ bivariate polynomial $H_{st}(x,y) =
  \sum_{i=0}^{k}\sum_{j=0}^{k} h_{k-i,k-j}^{(st)}T_i(y)T_j(x)$. 
It holds
  \[h^{(st)}_{k-i,k-j} = \left(\frac{2}{\pi}\right)^2\int_{-1}^1 \int_{-1}^1
  \frac{H_{st}(x,y)T_i(y)T_j(x)}{\sqrt{(1-x^2)(1-y^2)}}dxdy,\] 
(for $i=k$ and $j=k$
the constant is $\frac{1}{\pi}$) and hence using $|T_i(x)|\leq 1$ on $[-1,1]$ we obtain
  \begin{align*}
    |h^{(st)}_{k-i,k-j}|&\leq \left(\frac{2}{\pi}\right)^2\max_{x\in[-1,1],y\in[-1,1]}|H_{st}(x,y)| \int_{-1}^1 \int_{-1}^1 \frac{1}{\sqrt{(1-x^2)(1-y^2)}}dxdy\\
    &=4\max_{x\in[-1,1],y\in[-1,1]}|H_{st}(x,y)|\\
    &\leq 8(e-1)k^2 \|P(\cdot)\|\|v(\cdot)\|,
  \end{align*}
  where we used~\eqref{eq:Hbound} for the last inequality.  Since this
  holds for every $(s,t)$ and $(i,j)$ we conclude that
  \[
  \|X\|_2\leq 8n(e-1)k^3 \|P(\cdot)\|\|v(\cdot)\|
  \]
  as required.  

 To bound $\|Y\|_2$ we use
   the fact that $Y$ is the B\'{e}zout block matrix
  $B(P,-vxI)$, and by an analogous argument we obtain the bound 
  \[
  \|Y\|_2\leq 8n(e-1)k^3 \|P(\cdot)\|\|v(\cdot)\|.
  \]
  This establishes~\eqref{eq:xysmall} 
  with 
$q_X(n,k)=q_Y(n,k)=8n(e-1)k^3$ 
, and we obtain 
  \begin{equation}    \label{eq:Lbound}
  \|L(\cdot)\|\leq 16n(e-1)k^3 \|P(\cdot)\|\|v(\cdot)\|.    
  \end{equation}




Substituting this into~\eqref{eq:rlam}
we obtain
\begin{align*}
  r_{\lambda_i}&\leq
\frac{k}{v(\lambda_i)}\frac{\|L(\cdot)\|}{\|P(\cdot)\|}
\leq \frac{k}{v(\lambda_i)}\frac{16n(e-1)k^3 \|P(\cdot)\|\|v(\cdot)\|}{\|P(\cdot)\|}. 
\end{align*}
With the choice $v = 1$ we have $v(\lambda_i) = \|v(\cdot)\|=1$,
which yields~\eqref{eq:condrate}. 
\end{proof}

Note that our discussion deals with the normwise condition number, as
opposed to the coefficientwise condition number as treated in~\cite{Higham_06_02}.
In practice, we observe that the eigenvalues of $L(\lambda)$ computed
via the QZ algorithm are sometimes less accurate than those of
$P(\lambda)$, obtained via QZ for the colleague
linearization~\cite{good1961colleague}, which is normwise
stable~\cite{NNvandooren14}.
The reason appears to be that the backward error resulting from the
colleague matrix has a special structure, but a precise explanation is
an open problem.

\section{Construction}\label{sec:construction}
We now describe an algorithm for computing $\mathbb{DL}$ pencils. 
The shift sum operation 
provides a means to obtain the $\mathbb{DL}$ pencil given the ansatz $v$. 
For general polynomial bases, however, the construction is not as trivial as for the monomial basis. 
We focus on the case where $\{\phi_i\}$ is an orthogonal polynomial basis, so that the multiplication matrix~\eqref{eq:defM} has tridiagonal structure. 
Recall that $F(x,y)$ and $G(x,y)$ satisfy the formulas~\eqref{eq:DLrelations},~\eqref{eq:F} and 
\eqref{eq:G}. Hence 
for $L(\lam) = \lambda X+Y\in \mathbb{DL}(P)$ with ansatz $v$, 
writing the bivariate equations in terms of their coefficient matrix expansions, we see that 
 $X$ and $Y$ need to satisfy the following equations:
defining $v = [v_{k-1},\ldots,v_0]^T$ to be the vector of coefficients of the ansatz, and setting 
\[
S = v\otimes \left[P_k,P_{k-1},\ldots,P_0\right]\qquad \hbox{and} \qquad T = v^T\otimes \left[P_k,P_{k-1},\ldots,P_0\right]^\B, 
\]
Note that $S$ and $T$ are the matrix representation of the functions $P(y)v(x)$ 
and $v(y)P(x)$ respectively. 
Hence by~\eqref{eq:G} we have 
\begin{equation}\label{eq:getY}
\begin{bmatrix} 0 \cr Y \end{bmatrix} M - M^T\begin{bmatrix}0 & Y\end{bmatrix} = TM - M^T S, 
\end{equation}
where $M$ is as in~\eqref{eq:defM}, the matrix representing the 
shift operation; recall that $M^{\mathcal{B}}=M^T$. Similarly, by~\eqref{eq:F} we have 
\begin{equation}\label{eq:getX}
XM = S - \begin{bmatrix}0 & Y \end{bmatrix}.
\end{equation}
Note that we have used the first equation of~\eqref{eq:DLrelations} instead of~\eqref{eq:F} to obtain an equation for $X$ because the former is simpler to solve. 
Now we turn to the computation of $X,Y$, which also explicitly shows that the pair 
$(X,Y)$ satisfying~\eqref{eq:getY},~\eqref{eq:getX}  is unique\footnote{We note that~\eqref{eq:getY} is a singular Sylvester equation, but if we force the zero structure in the first block column in $[0\ Y]$ then the solution becomes unique. 
}. 
We first solve~\eqref{eq:getY} for $Y$. Recall that $M$ in~\eqref{eq:Morth} is block tridiagonal, the $(i,j)$ block being $m_{i,j}I_n$. 
Defining $R = TM - M^T S$ and 
denoting by $Y_i,R_{i}$ the $i$th block rows of $Y$ and $R$ respectively, 
the first block row of~\eqref{eq:getY} yields $m_{1,1}Y_1 = -R_{1}$, hence 
$Y_1 = -\frac{1}{m_{1,1}}R_{1}$ (note that $m_{i,i}\neq 0$ because the polynomial basis is degree-graded). The second block row of~\eqref{eq:getY} gives
$Y_1M-(m_{1,2}Y_1+m_{2,2}Y_2) = R_{2}$, hence
$ Y_2 = \frac{1}{m_{2,2}}(Y_1M-m_{1,2}Y_1-R_{2})$.
Similarly, from the $i(\geq 3)$th block row of~\eqref{eq:getY} we get
\begin{equation}\nonumber
 Y_i = \frac{1}{m_{i,i}}(Y_{i-1}M-m_{i-2,i}Y_{i-2}-m_{i-1,i}Y_{i-1}-R_{i}),
\end{equation}
so we can compute $Y_i$ for $i=1,2,\ldots,n$ inductively. 
Once $Y$ is obtained, $X$ can be computed easily by~\eqref{eq:getX}.
The complexity is 
$\mathcal{O}((nk)^2)$, noting that $Y_{i-1}M$ can be computed with $\mathcal{O}(n^2k)$ cost. 
In Section~\ref{appendix} we provide a {\sc Matlab} code that computes  $\mathbb{DL}(P)$  for any orthogonal polynomial basis.

If $P(\lam)$ is expressed in the monomial basis we have (see~\cite[eq.~2.9.3]{Bini_94_01} for scalar polynomials) 
\[
L(\lam) = \begin{bmatrix}P_{k-1}&\ldots&P_0\cr\vdots&\iddots\cr P_0\cr\end{bmatrix}\!\!\begin{bmatrix}\hat{v}_kI_n&\ldots&\hat{v}_1I_n\cr&\ddots&\vdots\cr &&\hat{v}_kI_n\end{bmatrix} - \begin{bmatrix}\hat{v}_{k-1}I_n&\ldots&\hat{v}_0I_n\cr\vdots&\iddots\cr\hat{v}_0I_n\cr\end{bmatrix}\!\! \begin{bmatrix}P_k&\ldots&P_1\cr&\ddots&\vdots\cr&&P_k\end{bmatrix},
\]
where $\hat{v}_i = \left(v_{i-1} - \lambda v_i\right)$.  This relation can be used to obtain expressions for the block matrices 
$X$ and $Y$. For other orthogonal bases the relation is more complicated. 

Matrix polynomials expressed in the Legendre or Chebyshev basis are of practical importance, 
for example, for a nonlinear eigenvalue solver based on Chebyshev interpolation~\cite{Effenberger_11_01}.
Following~\cite[Table~5.2]{Mackey_05_01}, in Table~\ref{tab:ChebTDLP} we depict 
three $\mathbb{DL}(P)$ pencils for the cubic matrix polynomial 
$P(\lam) = P_3T_3(\lam) + P_2T_2(\lam) + P_1T_1(\lam) + P_0T_0(\lam)$, where $T_j(\lam)$ is the $j$th Chebyshev polynomial.

\begin{table}[h]
\centering
\caption{Three instances of pencils in $\mathbb{DL}(P)$ and their linearization condition for the cubic matrix polynomial $P(\lam) = P_3T_3(\lam) +
P_2T_2(\lam) + P_1T_1(\lam) + P_0T_0(\lam)$, expressed in the Chebyshev basis of the first kind. These three pencils 
form a basis for the vector space $\mathbb{DL}(P)$.}
\label{tab:ChebTDLP}
\footnotesize
\begin{tabular}{|c|c|c|}
\hline
\rule{0pt}{3ex}$v$ & $L(\lam)\in\mathbb{DL}(P)$ for given $v$ & Linearization condition \\[3pt]
\hline
& & \\
$\begin{bmatrix} 1\\0 \\0\end{bmatrix} $& $\lam\!\begin{bmatrix}
2P_3\!&0&0\\0&\!2P_3-2P_1\!&-2P_0\\0&-2P_0&\!P_3-P_1\!\end{bmatrix} +
\begin{bmatrix}P_2&\!P_1-P_3\!&P_0\\\!P_1-P_3\!&2P_0&\!P_1-P_3\!\\P_0&\!P_1-P_3\!&P_0\end{bmatrix} $
& 
\parbox{3cm}{\centering
$\det(P_0 + \frac{-P_3+P_1}{\sqrt{2}})\neq 0 $\\\vspace*{3pt}
$\det(P_0 - \frac{-P_3+P_1}{\sqrt{2}})\neq 0 $
}
\\& &\\
\hline
& &\\
$\begin{bmatrix} 0\\1\\0\end{bmatrix} $& $\lam\!\begin{bmatrix}
0&2P_3&0\\2P_3&2P_2&2P_3\\0&2P_3&P_2-P_0\end{bmatrix} +
\begin{bmatrix}-P_3&0&-P_3\\0&P_1-3P_3&P_0-P_2\\-P_3&P_0-P_2&-P_3\end{bmatrix} $
& 
\parbox{3cm}{\centering
$\det (-P_2+P_0)\neq 0$\\\vspace*{3pt}
$\det (P_3)\neq 0$ 
}
\\& &\\
\hline
& &\\
$\begin{bmatrix} 0\\0\\1\end{bmatrix} $& $\lam\!\begin{bmatrix}
0&0&2P_3\\0&4P_3&2P_2\\2P_3&2P_2&P_1+P_3\end{bmatrix} +
\begin{bmatrix}0&-2P_3&0\\-2P_3&-2P_2&-2P_3\\0&-2P_3&P_0-P_2\end{bmatrix}$
& $\det (P_3)\neq 0$
\\&&\\
\hline 
\end{tabular}
\end{table} 
\normalsize

\subsection{{\sc Matlab} code for $\mathbb{DL}(P)$} \label{appendix}
The formulae~\eqref{eq:F} and~\eqref{eq:G} can be used to construct any pencil in $\mathbb{DL}(P)$ without 
basis conversion, which can be numerically important~\cite{Amirasiani_09_01, Nakatsukasa_13_01}.  We provide a {\sc Matlab} code that constructs
pencils in $\mathbb{DL}(P)$ when the matrix polynomial is expressed in any orthogonal basis. If $P(\lambda)$ is expressed 
in the monomials then ${\tt a=[ones(k,1)];}$ ${\tt b = zeros(k,1);}$ ${\tt c = zeros(k,1);}$ and if expressed in the Chebyshev basis then 
${\tt a=[ones(k-1,1);2]/2;}$ ${\tt b = zeros(k,1);}$ ${\tt c = ones(k,1)/2;}$.
\vspace{.2cm}
\begin{verbatim}
function [X Y] = DLP(P,v,a,b,c)
%DLP constructs the DL pencil with ansatz vector v. 
% [X,Y] = DLP(P,v,a,b,c) returns the DL pencil lambda*X + Y 
%    corresponding to the matrix polynomial with coefficients P in an 
%    orthogonal basis defined by the recurrence relations a, b, c.

[n m] = size(P); k=m/n-1; s=n*k;              % matrix size & degree  
M = spdiags([a b c],[0 1 2],k,k+1);
M = kron(M,eye(n));                           % multiplication matrix

S = kron(v,P);
for j=0:k-1, jj=n*j+1:n*j+n; P(:,jj)=P(:,jj)';end % block transpose
T = kron(v.',P'); R=M'*S-T*M;                 % construct RHS

% The Bartels-Stewart algorithm on M'Y+YM=R
X = zeros(s); Y=X; ii=n+1:s+n; nn=1:n;        % useful indices
Y(nn,:)=R(nn,ii)/M(1); X(nn,:)=T(nn,:)/M(1);  % 1st column of X and Y
Y(nn+n,:)=(R(nn+n,ii)-M(1,n+1)*Y(nn,:)+Y(nn,:)*M(:,n+1:s+n))/M(n+1,n+1);
X(nn+n,:)=(T(nn+n,:)-Y(nn,:)-M(1,n+1)*X(nn,:))/M(n+1,n+1); % 2nd cols

for i = 3:k                                     % backwards subs
    ni=n*i; jj=ni-n+1:ni; j0=jj-2*n; j1=jj-n;   % useful indices
    M0=M(ni-2*n,ni); M1=M(ni-n,ni); m=M(ni,ni); % consts of 3-term
    Y0=Y(j0,:); Y1=Y(j1,:); X0=X(j0,:); X1=X(j1,:); % vars in 3-term
    Y(jj,:)=(R(jj,ii)-M1*Y1-M0*Y0+Y1*M(:,n+1:s+n))/m; 
    X(jj,:)=(T(jj,:)-Y1-M1*X1-M0*X0)/m;         % use Y to solve for X
end
\end{verbatim}
\vspace{.2cm}

%
%

\subsection*{Acknowledgments}
We wish to thank Nick Higham, Fran\c{c}oise Tisseur, and Nick Trefethen for their support and insightful comments. 
We are grateful to the anonymous referees for their careful reading of the manuscript that lead to an improved presentation. 
We finally thank Leiba Rodman for his words of appreciation and encouragement, 
and we posthumously dedicate this article to him.

%
%
%
%
%
%
%
%
\bibliographystyle{abbrv}
\bibliography{M4revnewSIMAXrev}

\end{document}